\documentclass{article}

\usepackage[style=numeric, backend=biber, maxbibnames=99]{biblatex} % biblatex package for managing bibliographies
\usepackage[colorlinks=true, linkcolor=red, urlcolor=blue, citecolor=blue]{hyperref}
\usepackage{graphicx} % Required for inserting images
\usepackage{amsmath}
\usepackage{amsthm}
\usepackage{amssymb}
\usepackage{listings}
\usepackage{color}
\usepackage{xcolor}
\usepackage{csquotes}
\usepackage[english]{babel}
\usepackage{tabularx}
\usepackage{booktabs} % For better looking tables
\usepackage{geometry} % For adjusting page margins if necessary
\usepackage{pdflscape}
\usepackage{appendix} % Required for appendix sections
\usepackage{mathdots}

\newcommand{\N}{\mat{N}}
\newcommand{\Nb}{\N}   % Basis matrix for xNx=0
\newcommand{\Rb}{\N}   % Basis matrix for R in CSTW
\newcommand{\Sf}{W}    % Stability function R(z)
\newcommand{\A}{\mat{A}}
\renewcommand{\b}{\vec{b}}
\newcommand{\R}{\mat{R}}
\newcommand{\X}{\mat{X}}
\renewcommand{\L}{\mat{L}}
\newcommand{\D}{\mat{D}}
\renewcommand{\P}{\mat{P}}
\newcommand{\F}{\mat{F}}
\renewcommand{\v}{\vec{v}}

\renewcommand{\u}{\vec{u}}

\newcommand{\veta}{\vec{\eta}}
\newcommand{\M}{\mat{M}}
\newcommand{\e}{\vec{e}}
\newcommand{\nb}{\vec{n}}
\newcommand{\I}{\mat{I}}
\newcommand{\B}{\mat{B}}
\newcommand{\leta}{\overline{\eta}}

\newtheorem{theorem}{Theorem}[section]
\newtheorem{remark}[theorem]{Remark}
\newtheorem{example}[theorem]{Example}

\newtheorem{lemma}[theorem]{Lemma}

\newcommand{\myremarkend}{$\spadesuit$}
\newcommand{\mat}[1]{\mbox{\boldmath$#1$}}
\renewcommand{\vec}[1]{\boldsymbol{#1}}

\newcounter{opt}
\newenvironment{opt}{\refstepcounter{opt}\begin{equation}\begin{array}{ll}}{\end{array}\tag{P\theopt}\end{equation}}

\title{Algebraic Conditions for Stability in Runge-Kutta Methods and Their Certification via Semidefinite Programming}
\author{
    Austin Juhl\thanks{Department of Mathematical Sciences, NJIT, Newark, NJ 07102, \newline \indent \hspace{2mm} aj659@njit.edu.}, 
    David Shirokoff\thanks{Department of Mathematical Sciences, NJIT, Newark, NJ 07102, \newline \indent \hspace{2mm} david.g.shirokoff@njit.edu.}
}
\date{May 2024}

\begin{document}

\maketitle

\begin{abstract}
In this work, we present approaches to rigorously certify $A$- and $A(\alpha)$-stability in Runge-Kutta methods through the solution of convex feasibility problems defined by linear matrix inequalities.  We adopt two approaches. The first is based on sum-of-squares programming applied to the Runge-Kutta $E$-polynomial and is applicable to both $A$- and $A(\alpha)$-stability.  In the second, we sharpen the algebraic conditions for $A$-stability of Cooper, Scherer,  T{\"u}rke, and Wendler to incorporate the Runge-Kutta order conditions.  We demonstrate how the theoretical improvement enables the practical use of these conditions for certification of $A$-stability within a computational framework.  We then use both approaches to obtain rigorous certificates of stability for several diagonally implicit schemes devised in the literature. 
\end{abstract}

%======================================================================

% --------------------------------------------------------------------------- 80

\medskip\noindent
% --------------------------------------------------------------------------- 80

%======================================================================
\section{Introduction}
%======================================================================

In recent years, sum-of-squares optimization and semidefinite programming have become valuable tools in developing rigorous certificates of stability in dynamical systems.  Such certificates are useful in developing reliable algorithms and software.  Examples range from 
stochastic linear-quadratic control \cite{YaoZhangZhou2001}, 
switched stability of nonlinear systems \cite{AhmadiJungers2013},
reinforcement learning \cite{JinLavaei2020}, 
stability in partial differential equations \cite{GahlawatValmorbida2017}, and 
applications for robotic control \cite{MajumdarAhmadiTedrake2013}.  
Additionally, control systems have benefited from analogous certification techniques, including verifying the stability of reinforcement learning policies \cite{DroriTeboulle2014} or variable step-size convergence bounds for gradient descent \cite{Grimmer2024}. In this work, we adopt a similar strategy to establish rigorous certificates for $A$- and $A(\alpha)$- stability in Runge-Kutta methods through the computational solution of linear matrix inequalities via semi-definite programming.

%======================================================================
\subsection{Background}
%======================================================================
Numerical stability is a critical property for a time-integration scheme.  In the context of Runge-Kutta (RK) methods applied to stiff differential equations, $A$-stability (or the related $A(\alpha)$-stability), introduced by Dahlquist \cite{Dahlquist1963}, is one of the most basic and practically important notions of stability. As a result of the reduced degrees of freedom in the Butcher coefficients, Runge-Kutta methods that are $A$-stable with stability functions that are near-optimal order approximations have been well characterized \cite{Butcher1977, Wanner1980, Hairer1982, Norsett1975} (cf. \cite{HairerWanner1981} for algebraically stable methods). However, many existing and newly developed Runge-Kutta methods in the literature admit a diagonally implicit structure and lie outside the application of these previous results. While one strategy for verifying $A$-stability is via \emph{Sturm sequences} \cite{HairerWanner1996}, our approach here is rooted in convex optimization.  

Dating back to the work of Dahlquist (e.g., \cite{Dahlquist1978}), it has been known that $A$-stability is equivalent to the RK stability function satisfying a particular convex feasibility problem --- i.e., up to a transformation, the stability function lies in the convex cone of \emph{positive functions} \cite[Chapter IV.5]{HairerWanner1996}.  Subsequent convex feasibility conditions for $A$-stability also include: (1) The RK $E$-polynomial lying in the convex cone of non-negative polynomials; and (2) Satisfying a set of algebraic conditions, which we refer to as the CSTW conditions, developed by Cooper \cite{cooper:1986}, Scherer and T{\"u}rke \cite{scherer_turke:1989}, and Scherer and Wendler \cite{scherer_wendler:1994}. The CSTW conditions form a feasibility problem over the convex cone of semidefinite matrices.  It is worth noting that semidefinite matrix conditions exist for stronger notions of stability, such as \emph{algebraic stability} (or the related concept of $B$-stability) for Runge-Kutta methods \cite[Chapter IV.12]{HairerWanner1996}, and $G$-stability \cite{dahlquist1975stability} for linear multistep methods.  

%======================================================================
\subsection{Contributions and Organization of the Paper}
%======================================================================
The focus of this work lies in the practical aspects of implementing and generating rigorous certificates of stability through solutions of convex feasibility problems for $A$- and $A(\alpha)$-stability.  The main contributions are twofold.  First, we provide a theoretical contribution that sharpens the CSTW conditions. The CSTW conditions do not take into account the fact that the stability function is a $p$th order approximation to the exponential --- which ultimately limits their practical use in providing rigorous certificates via computer-assisted means.  Our main theoretical result (Theorem~\ref{Thm:CSTWModified}) modifies the CSTW conditions to account for the RK order conditions, which then enables the rigorous certification of stability via computational approaches.  In addition to sharpening the CSTW conditions, we provide details from sum-of-squares optimization for certifying $A$- and $A(\alpha)$-stability via testing the non-negativity of the $E$-polynomial.

Our second contribution provides rigorous certificates of stability for several recently devised schemes in the literature. These schemes include the 7-stage 4th order and 12-stage 5th order schemes from \cite{BiswasKetchesonSeiboldShirokoff2023} as well as the 9-stage 6th order and 13-stage 8th order schemes from \cite{AlamriKetcheson2024}.

The paper is organized as follows: \S\ref{sec:RKBackground} and \S\ref{sec:LMI_Background} cover background material on Runge-Kutta methods and linear matrix inequalities, respectively. Linear matrix inequalities for $A$- and $A(\alpha)$- stability of Runge-Kutta methods are then developed in \S\ref{sec:algcond}. This section also includes our main theoretical result along with illustrative examples. We then turn our attention in \S\ref{sec:Examples} to constructing rigorous certificates of $A$- and $A(\alpha)$-stability via a computer-assisted solution to the associated linear matrix inequalities.  A summary and outlook are provided in \S\ref{sec:Conclusion}.  The Appendix and companion Github repository provide certificates of stability for the schemes we examine as well as the supporting numerical code.

%======================================================================
\section{Runge-Kutta Background}\label{sec:RKBackground}
%======================================================================
Numerical time integration of an ordinary differential equation ($\mathbb{V} = \mathbb{R}^m$ or $\mathbb{C}^m$)
\begin{equation}\label{eq:IVP}
    \u^{\prime}(t)  = f\big(t,\u(t) \big) \, , \quad \u(0) = \u_{0} \, ;  \quad \u\in \mathbb{V} \, , 
    \quad f : \mathbb{R}\times \mathbb{V} \to \mathbb{V} \, ,
\end{equation}
results in a discrete-in-time dynamical system --- one of the most common being Runge-Kutta (RK) methods. Runge-Kutta methods, discretize \eqref{eq:IVP} with $s$ stages as
    \begin{subequations}\label{eq:RK_methods}
        \begin{align*}
            \vec{h}_i & = \u_n + \Delta t \, \sum_{j = 1}^{s} a_{ij} \, f(t_n+c_j \Delta t, \vec{h}_j)\;, \quad i = 1,2, \ldots,s\; \\ 
            \u_{n+1} & = \u_{n}+\Delta t \, \sum_{j = 1}^{s} b_j \, f(t_n+c_j \Delta t, \vec{h}_j)\; \, ,
        \end{align*}
    \end{subequations}
    where $\u_n \approx \u(t_n)$ and $t_n$ the $n$th time step. The RK scheme is defined by coefficients
    \begin{equation*}
        \A = [a_{ij}]_{i,j=1}^s \, , \qquad \b = [b_1, \ldots, b_s]^T \, , \quad\textrm{and}\quad \vec{c} = [c_1, \ldots, c_s]^T := \A \e \, ;\quad\textrm{where}\quad \e=[1,\ldots,1]^T\, .
    \end{equation*}
    
%======================================================================
\subsection{Linear Stability}\label{subsec:linstability}
%======================================================================
For the scalar linear case where $f(\u) = \lambda \u$ with $\lambda \in \mathbb{C}$ and fixed $\Delta t$, the Runge-Kutta dynamics \eqref{eq:RK_methods} are expressed as:
\begin{align*}
\u_{n+1} = \Sf(z) \u_n, \quad \text{where} \quad z := \lambda \Delta t \, .
\end{align*}
Here $\Sf(z)$ is the \emph{stability function} of the scheme given by:
\begin{equation}\label{eq:stabilityfunction1}
\Sf(z) := 1 + z \b^T (\I - z \A)^{-1} \e = \frac{\det(\I-z\A+z \e \b^T )}{\det(\I-z\A)} = \frac{N(z)}{D(z)}.
\end{equation}
In equation \eqref{eq:stabilityfunction1}, $N(z)$ and $D(z)$ represent polynomials of degree at most $s$, sharing no common factors, and $\I$ is the identity matrix.  

A method exhibits a \emph{degenerate} stability function if $\deg N \leq s-1$ and $\deg D \leq s-1$; otherwise, the stability function is considered non-degenerate.  Degenerate stability functions occur when $\det(\I-z\A+z \e \b^T )$ and $\det(\I - z\A)$ share a common root. It is worth noting that degenerate stability functions can appear in practice, for instance, in constructing RK schemes that avoid order reduction (e.g., \cite{BiswasKetchesonSeiboldShirokoff2023}). 

The dynamics \eqref{eq:RK_methods} are \emph{stable} for a given $z$ if $|\Sf(z)| \leq 1$. We will primarily be concerned with numerical schemes $(\A,\b)$ that are $A$-stable, i.e., those where
\begin{alignat*}{2} 
    \textrm{$A$-stability:} &\qquad 
    |\Sf(z)| \leq 1  \qquad \textrm{for} \qquad z \in \mathbb{C}_- = \{ z \in \mathbb{C} \; : \; \textrm{Re} \, z \leq 0 \} \, .
\end{alignat*}
The $A$-stability criteria ensures that the discrete dynamics \eqref{eq:RK_methods} are stable whenever the linear ODE \eqref{eq:IVP} is stable. For schemes that are not $A$-stable, it is useful to characterize the largest $\alpha > 0$ for which the RK dynamics are stable for all $z$ in the sector $S_{\alpha}$ opening into the left-half plane with angle $\alpha$
\begin{alignat*}{2} 
    A(\alpha)\textrm{-stability:} & \qquad &  |\Sf(z)| \leq 1  \qquad \textrm{for} \qquad & z \in S_{\alpha} = \{ z \in \mathbb{C} \, : \, |\arg(-z)| \leq \alpha \}  \, . 
\end{alignat*}
Note that $A$-stability is equivalent to $A(\alpha)$-stability with $\alpha = \frac{\pi}{2}$.
Both $A$- and $A(\alpha)$- stability can be recast in terms of the non-negativity of the $E$-polynomial \cite{HairerWanner1996}. 

In a slight generalization from the standard definition, to allow for $\alpha \neq \frac{\pi}{2}$, the (generalized) $E$-polynomial is
\begin{align}\label{Def:genEpoly}    
    E(y;\alpha) = |D(ye^{-i\alpha})|^2 - |N(ye^{-i\alpha})|^2 \, 
    \quad \textrm{where} \quad 0 < \alpha \leq \frac{\pi}{2} \,, \; y \in \mathbb{R} \, .
\end{align}

A scheme is then $A(\alpha)$-stable if (for $\alpha=\frac{\pi}{2}$ see \cite[Chapter IV.3]{HairerWanner1996}) $\Sf(z)$ is analytic in the interior of $S_{\alpha}$ and 
\begin{align}\label{Eq:Nonneg_Epoly}
    E(y; \alpha) \geq 0 \qquad \textrm{for all} \qquad y \geq 0 \, . 
\end{align}
The condition \eqref{Eq:Nonneg_Epoly} guarantees \(A(\alpha)\)-stability by ensuring \(|\Sf(z)| \leq 1\) for all \(z\in\partial S_{\alpha}\). As \(\Sf(z)\) is analytic within the interior of \(S_{\alpha}\), the maximum modulus principle implies that \(|\Sf(z)|\) is maximized on \(\partial S_{\alpha}\), confirming \(A(\alpha)\)-stability.

%======================================================================
\subsection{Accuracy and Order Conditions}
%======================================================================
For an RK scheme to achieve \emph{(classical) order $p$} on linear, autonomous problems, the stability function must approximate the exponential function to order $p$, such that $\Sf(z) = e^z + \mathcal{O}(z^{p+1})$ as $z \rightarrow 0$. This approximation is achieved provided the RK coefficients $(\A,\b)$ satisfy the \emph{tall-tree} order conditions of order $p$:
\begin{equation}\label{eq:ordercondition}
    \b^T \A^{j-1} \, \e = \frac{1}{ j!} \,  
	\quad \text{for} \quad
	1\leq j \leq p \, ,
\end{equation}
as outlined in \cite{HairerNorsettWanner1993}. Additional RK order conditions, i.e., the non-tall-tree conditions, are further required to achieve accuracy of order $p$ on general (nonlinear) ODEs.  

If $\alpha = \frac{\pi}{2}$, the $p$th order conditions \eqref{eq:ordercondition} imply that $E(y; \frac{\pi}{2})$ admits a factor of $y^{2j}$ \cite[Chapter IV.3]{HairerWanner1996}:
\begin{align}\label{EPoly_Asymptotics}
    E\left(y; \frac{\pi}{2} \right) = \mathcal{O}(y^{2j} ) \, , 
    \qquad \textrm{as} 
    \quad y \rightarrow 0 
    \quad \textrm{where}, 
    \quad j \geq \left\lfloor \frac{p}{2} \right\rfloor+1 \, . 
\end{align}
In general for $\alpha < \frac{\pi}{2}$, the polynomial (as implied by the calculations in \cite{cooper:1986})
\begin{align}\label{EPoly_Asymptotics2}
    E\left(y; \alpha \right) = \mathcal{O}(y) \, , 
    \qquad \textrm{as} 
    \quad y \rightarrow 0 \, .
\end{align}

%===============================================================================
\section{Linear Matrix Inequalities Background}\label{sec:LMI_Background}
%===============================================================================
Certifying numerical stability will reduce to a semidefinite feasibility program involving linear matrix inequalities (LMIs). This section reviews LMIs and their relationships to convex feasibility as well as non-negative polynomials.

%===============================================================================
\subsection{Feasibility and Convexity of Linear Matrix Inequalities} 
\label{subsec:Convex_Opt}
%===============================================================================
Given $\P, \N_1, \, \ldots, \, \N_n \in \mathbb{S}^n$, the set of $n \times n$ real symmetric  matrices, a \emph{linear matrix inequality} (LMI) is defined as:
\begin{align}\label{Def:LMI}
    \F(\veta) := \P + \sum_{j = 1}^d \eta_j \N_j \succeq 0 \,  ,
\end{align}
where $\F \succeq 0$ indicates that $\F$ is positive semi-definite ($\F \succ 0$ is positive definite). The LMI \eqref{Def:LMI} is \emph{feasible} if there exists a vector $\veta$ such that $\F(\veta) \succeq 0$; otherwise the LMI is \emph{infeasible}. The linearity of $\F(\veta)$ ensures that the set: 
\begin{align*}%\label{Def:LMIFeasible}
    \mathcal{C} = \{ \veta \in \mathbb{R}^d \; : \; \F(\veta) \succeq 0\} \, , 
\end{align*}
containing all $\veta$ that satisfy the LMI \eqref{Def:LMI}, is convex. Thus, assessing the feasibility of $\F(\veta)$ --- determining whether $\mathcal{C}$ is non-empty --- is a convex feasibility problem and can be solved via semidefinite programming. 

Due to the matrix structure of \eqref{Def:LMI}, the feasible set $\mathcal{C}$ may lie in an affine plane having dimension less than $d$, potentially resulting in $\mathcal{C}$ having an empty interior. To quantify this feature, the dimension for the set $\mathcal{C}$ may be characterized in terms of the affine hull: 
\begin{align*}
    \textrm{aff}(\mathcal{C}) = \{ \mu_1 \veta_1 + \mu_2 \veta_2 \; : \; \mu_1, \mu_2 \in \mathbb{R}, \; \textrm{and }  \veta_1, \veta_2 \in \mathcal{C}\} =  
    \veta_0 + \mathbb{V} \, ,     
\end{align*}
where $\veta_0 \in \mathcal{C}$ and $\mathbb{V}$ is a vector space. The dimension of $\mathcal{C}$ is defined as $\dim\mathbb{V}.$

%===============================================================================
\subsection{Non-negative Polynomials as Linear Matrix Inequalities}\label{subsec:BackgrounPolyLMI}
%===============================================================================
Here, we review characterizing the non-negativity of a polynomial through linear matrix inequalities.

Let $\mathbb{R}[x]$ denote the set of single-variable polynomials with real coefficients. Two convex cones within $\mathbb{R}[x]$ include the set of non-negative polynomials, satisfying $p(x) \geq 0$ for all $x \in \mathbb{R}$, and the set of sum-of-squares (SOS) polynomials, where each polynomial can be expressed as $\sum_{j=1}^{\ell} q_j^2(x)$ for some $\{ q_j\}_{j =1}^{\ell} \in \mathbb{R}[x]$. In one dimension, these two cones coincide, meaning that a polynomial $p(x) \in \mathbb{R}[x]$ is non-negative for all real numbers if and only if it can be decomposed into a sum-of-squares \cite[Theorem 2.5]{Lasserre2010}. For general multivariate polynomials, every SOS polynomial is non-negative; however, the reverse is not true. 

Determining whether a polynomial is SOS and non-negative can be formulated as an LMI. Consider the subspace of symmetric matrices $\mathcal{N}_{m}$ defined by:
\begin{align*}
    \mathcal{N}_{m} := \big\{  \N \in \mathbb{S}^{m} \; : \; \vec{y}^T \N \vec{y} = 0  \big\} \, , \quad\textrm{where}\quad \vec{y}=[1,y,\ldots,y^{m-1}] \, .
\end{align*}
The components $(n_{i,j})_{i,j=1}^{m}$ of a symmetric matrix $\N \in \mathcal{N}_{m}$ must satisfy:
\begin{align*}
    \sum_{i + j=r}^{m} n_{i,j}  = 0 \qquad \textrm{for} \quad r = 2, \; 3, \;\ldots, \; 2m \,.    
\end{align*}
Any polynomial $p\in \mathbb{R}[x]$ can then be expressed non-uniquely in factorized form as:
\begin{align}\label{Eq:PolyForm}
    p(x) = p_0 + p_1 y + \ldots + p_{2m-2} \, y^{2m-2} = \vec{y}^T (\P + \N) \vec{y} \, , 
\end{align}
where $\N \in \mathcal{N}_{m}$ and
\begin{align*}%\label{Def:P_Mat}
    \P = 
        \begin{pmatrix}
        p_0 & \tfrac{1}{2} p_1 & & \\
        \tfrac{1}{2} p_1 & p_2 & \ddots \vphantom{\ddots} &  \\
        & \ddots & \ddots & \tfrac{1}{2}  p_{2m-3} \\  
        & &  \tfrac{1}{2} p_{2m-3} & p_{2m-2} 
    \end{pmatrix}    \in \mathbb{R}^{m \times m} \, .
\end{align*}
    The polynomial $p$, described in \eqref{Eq:PolyForm}, is a sum-of-squares if and only if there exists an $\N \in \mathcal{N}_{m}$ such that $\P + \N \succeq 0$. If $\mat{Q}^T \mat{Q} = \P + \N$ is a Cholesky factorization, then defining $q_j(y) = \e_j^T \mat{Q} \vec{y}$, where $\e_j$ is the $j$th unit vector, admits 
\begin{align*}
    p(y) = \| \mat{Q} \vec{y}\|^2 = \sum_{j=1}^{m} q_j^2(y) \, . 
\end{align*}Conversely, if $p$ is an SOS, the coefficients of $q_j$ define a positive definite matrix $(\P+\N)$. 

To solve subsequent SDPs, it is useful to define a basis for $\mathcal{N}_{m}$, which has dimension 
\begin{equation}\label{eq:d}
    d := \frac{1}{2}(m-1)(m-2).
\end{equation}
First, consider the index set 
\begin{align*}%\label{Def:SetSk}
    \mathcal{S}_{m} := \big\{  (i,j) \in \mathbb{N}^{2} \; | \; 1 \leq i \leq m-2 \, ,  i + 2 \leq j \leq m \big\} \, \qquad 
    \textrm{for} \qquad m \geq 1 \,, 
\end{align*}
which has $d$ elements and is empty for $m = 1,2$. A basis for $\mathcal{N}_{m}$ is given by $\{\Nb_l\}_{l = 1}^{d}$:

\begin{equation*}%\label{Def:Gij}
\Nb_{l}=\e_i\e_j^T+\e_j\e_i^T-\e_{\left\lfloor\frac{i+j}{2}\right\rfloor}\e_{\left\lceil\frac{i+j}{2}\right\rceil}^T-\e_{\left\lceil\frac{i+j}{2}\right\rceil}\e_{\left\lfloor\frac{i+j}{2}\right\rfloor}^T \, ,
\end{equation*}
where 
\begin{align*}
    (i,j)\in\mathcal{S}_{m} \qquad  \textrm{and} \qquad 
    l=m(i-1)-\frac{1}{2}i(i+3) +j \, . 
\end{align*}
For example, \( m=3 \) has the basis matrix:
\[
\Nb_1 = \begin{bmatrix}
0 & 0 & 1 \\
0 & -2 & 0 \\
1 & 0 & 0
\end{bmatrix},
\]
while \( m=4 \) has three basis matrices:
\[
\Nb_1 = \begin{bmatrix}
0 & 0 & 1 & 0 \\
0 & -2 & 0 & 0 \\
1 & 0 & 0 & 0 \\
0 & 0 & 0 & 0
\end{bmatrix}, \quad
\Nb_2 = \begin{bmatrix}
0 & 0 & 0 & 1 \\
0 & 0 & -1 & 0 \\
0 & -1 & 0 & 0 \\
1 & 0 & 0 & 0
\end{bmatrix}, \quad
\Nb_3 = \begin{bmatrix}
0 & 0 & 0 & 0 \\
0 & 0 & 0 & 1 \\
0 & 0 & -2 & 0 \\
0 & 1 & 0 & 0
\end{bmatrix}.
\]

With these notations, a polynomial $p(y)$, given by \eqref{Eq:PolyForm}, is non-negative if and only if there exists $\veta \in \mathbb{R}^{d}$ such that the following LMI is satisfied:
\begin{align}\label{Eq:SOS_LMI}
    \F(\veta) := \P + \sum_{l=1}^{d} \eta_l\, \Nb_l \succeq 0 \, .
\end{align}

%==================================================================================================
\section{$A$- and $A(\alpha)$-Stability as Linear Matrix Inequalities}\label{sec:algcond}
%==================================================================================================
To certify the stability of Runge-Kutta methods, we utilize linear matrix inequalities in two approaches. The first approach leverages the non-negativity of the (generalized) $E$-polynomial, suitable for both $A$- and $A(\alpha)$-stability. The second approach uses the algebraic conditions for $A$-stability established by Cooper \cite{cooper:1986}, Scherer and Türke \cite{scherer_turke:1989}, and Scherer and Wendler \cite{scherer_wendler:1994}. Our main theoretical contribution sharpens these algebraic conditions, which enables their practical use within an SDP framework.

%==================================================================================================
\subsection{The $E$-polynomial LMI for $A(\alpha)$-stability}\label{subsec:EpolyLMI}
%==================================================================================================
An LMI for $A(\alpha)$-stability follows immediately by applying the LMI for general non-negative polynomials, as described in \S~\ref{subsec:BackgrounPolyLMI}, to $E(y ; \alpha)$ in \eqref{Eq:Nonneg_Epoly}.

In particular, based on the asymptotics \eqref{EPoly_Asymptotics}--\eqref{EPoly_Asymptotics2}, let $F(y)$ be the polynomial from which the largest even monomial (ensured by the order conditions) has been factored out:
\begin{alignat*}{2}%\label{Eq:p_Def}
    \textrm{If} \; \alpha &= \frac{\pi}{2} \; &: \qquad F(y) &:= y^{-2\kappa} \, E(y; \frac{\pi}{2}), \quad \kappa = \left\lfloor \frac{p}{2} \right\rfloor + 1 \,  , \\ \nonumber 
    \textrm{If} \; \alpha &< \frac{\pi}{2} \; &: \qquad F(y) &:= y^{-2} \, E(y^2; \alpha)  \,  .
\end{alignat*}
In both cases, $F(y)$ is an even polynomial:
\begin{align}\label{Eq:CoefficientsOfp}
    F(y) = p_0 + p_2 \, y^2 + \ldots + p_{2m-2} \, y^{2m -2} \, , 
\end{align}
where the coefficients of $F$ are polynomial functions of the RK scheme coefficients and, in the case of $\alpha<\tfrac{\pi}{2}$, polynomial functions of $\beta:=\cos(\alpha)$.

Combining the polynomial $F(y)$ and the LMI \eqref{Eq:SOS_LMI}, it follows: 
\begin{lemma}\label{E_poly_lemma}
    A scheme is $A(\alpha)$-stable if: 
\begin{enumerate}
    \item $\A$ has eigenvalues outside $S_{\alpha}$ (so that $W(z)$ is analytic in $S_{\alpha}$); and
    \item The LMI \eqref{Eq:SOS_LMI} is feasible for the $E$-polynomial \eqref{Eq:CoefficientsOfp}. 
\end{enumerate}
\end{lemma}

%==================================================================================================
\subsection{The CSTW Algebraic Conditions and LMI for $A$-stability} 
%==================================================================================================
Closely related to the $E$-polynomial approach is an LMI based on algebraic conditions for $A$-stability, laid out in a line of work by Cooper \cite{cooper:1986}, Scherer and T{\"u}rke \cite{scherer_turke:1989} and Scherer and Wendler \cite{scherer_wendler:1994}.  

Cooper \cite{cooper:1986} initially established sufficient algebraic conditions for $A$-stability by factorizing the $E$-polynomial into a quadratic form. Scherer and T{\"u}rke \cite{scherer_turke:1989} later re-derived almost identical conditions by applying the Kalman-Yakubovich-Popov (KYP) Lemma to $\Sf(z)$; this further showed necessity (in addition to sufficiency) of the algebraic conditions of Cooper. Subsequent work by Scherer and Wendler \cite{scherer_wendler:1994} provided even more general algebraic conditions applicable to degenerate stability functions.

\begin{theorem}\label{SW:1994} (Scherer--Wendler, Theorem 6.1 in \cite{scherer_wendler:1994}) 
Let $\M$ be any matrix whose column space is equal to the span of $[\e, \A\e, \A^2\e, \ldots, \A^{s-1} \e]$. The function $\Sf(z)$ in \eqref{eq:stabilityfunction1} is $A$-stable if and only if there exists a matrix $\R\in\mathbb{S}^s$ such that
    \begin{align}\label{SW:AlgCondition}
        \begin{cases} 
            \R \, \e=\b \, , \\ 
            \X = \R \A + \A^T \R - \b \b^T \, ,\\
            \M^T \R \M \succeq 0 \, , \\
            \M^T \X \M \succeq 0 \, .         
        \end{cases}
    \end{align} 
\end{theorem}
We refer to \eqref{SW:AlgCondition} as the CSTW conditions.

\begin{remark}
    When $\Sf(z)$ is nondegenerate, the matrix $\M$ can be the identity matrix $\I$.  If $\Sf(z)$ is degenerate, then \eqref{SW:AlgCondition} with $\M = \I$ provides a sufficient condition for $A$-stability, but it may not be necessary.  In the degenerate case, $\M$ may be chosen as $[\e, \A\e, \ldots, \A^{r-1} \e]$, where $r$ is the smallest number for which $\A^r \e$ can be expressed in terms of the vectors $[\e, \A\e, \ldots, \A^{r-1} \e]$. 
\end{remark}
\begin{remark}[Degenerate stability functions]
    While degenerate stability functions may seem uncommon, they do appear in the literature. Recently, Runge-Kutta schemes with degenerate stability functions have found application in avoiding order reduction via the weak stage order conditions (also referred to as parabolic order conditions) \cite{BiswasKetchesonSeiboldShirokoff2023, BiswasKetchesonSeiboldShirokoff2024, BiswasKetchesonRobertsSeiboldShirokoff2024}.
    \myremarkend    
\end{remark}

The conditions \eqref{SW:AlgCondition} define a convex set for $\R$ and are readily converted into an LMI by parameterizing the equality constraints. Let
\begin{align*}
    \B := \textrm{diag}\begin{bmatrix}
        b_1, & b_2, & \ldots, & b_s 
    \end{bmatrix} \, , 
\end{align*}
and
\begin{align*}
    \Rb_{ij} := \nb_{ij} \nb_{ij}^T 
    \qquad
    \textrm{where} 
    \qquad 
    \nb_{ij} = \e_i - \e_j \, , 
\end{align*}
for $1 \leq i <  j \leq s$ with $\e_i$ being the $i$th unit vector.  Then by construction, $\Rb_{ij}$ is a basis for the vector space $\Rb_{ij}\e = 0$ (cf. \cite{scherer_wendler:1994}), and $\R$ has the form:
\begin{align}\label{eq:basisR}
    \R&= \B + \Rb(\veta) \, \quad \textrm{where} \quad 
    \Rb(\veta) = \sum_{i=1}^{s-1}\sum_{j=i+1}^s\eta_{ij} \Rb_{ij} \, .
\end{align}

Condition \eqref{SW:AlgCondition} in LMI form reads:
\begin{align}\label{LMI:CSTW}
    \begin{bmatrix}
        \B & \mat{0} \\
        \mat{0} & \B\A + \A^T \B - \b \b^T
    \end{bmatrix} + 
    \sum_{i=1}^{s-1}\sum_{j=i+1}^s\eta_{ij} \begin{bmatrix}
        \Rb_{ij} & \mat{0} \\
        \mat{0} & \Rb_{ij}\A + \A^T \Rb_{ij}
    \end{bmatrix} \succeq  0  \, . 
\end{align}

The CSTW LMI \eqref{LMI:CSTW} and the $E$-polynomial approach in Lemma \ref{E_poly_lemma} both ensure $A$-stability for $\Sf(z)$. A key difference between the two approaches lies in the incorporation of the order conditions. Satisfying the order conditions removes lower order terms in the $E$-polynomial. In contrast, the CSTW LMI \eqref{LMI:CSTW} does not account for the order conditions, which turns out to be important in practice.

%===============================================================================
\subsection{Main Lemma: Sharpening the CSTW Conditions}\label{sec:ModCSTW}
%===============================================================================

The authors in \cite{scherer_wendler:1994} observe (though do not resolve) that zero eigenvalues of $\X$ may limit the practical application of the algebraic conditions \eqref{SW:AlgCondition}. The following Lemma shows that the order conditions result in $\X$ always admitting a family of zero eigenvalues whenever the CSTW LMI is feasible.

\begin{lemma}\label{thm:nullvec}
    Let $(\A,\b)$ satisfy the order conditions \eqref{eq:ordercondition} with order $p\geq2$. If $\R \in \mathbb{S}^s$ satisfies $\R \e = \b$ and $\X \succeq 0$ where
    \begin{align*}%\label{Def:XMat}
        \X := \R\A+ \A^T \R - \b \b^T \, ,
    \end{align*}
    then $\X$ has the following null vectors:    
    \begin{align*}
        \X \A^{j-1} \e = 0 \qquad \textrm{for} \qquad 1 \leq j \leq \left\lfloor\frac{p}{2}\right\rfloor \, .
    \end{align*}    
\end{lemma}

The proof utilizes the fact that for any matrix $\X \succeq 0$, 
\begin{align*}
    \textrm{if} \qquad \v^{\,T} \X \v = 0\,, \qquad \textrm{then} \qquad \X \v = 0 \, . 
\end{align*}
For instance, the Cholesky factorization $\X = \mat{Q}^T \mat{Q}$ shows that $\v^T \X \v = \| \mat{Q} \v \|^2 = 0$.

\begin{proof} To simplify notation, let 
\begin{align*}
    \v_j = \A^j \e \qquad j \geq 0 \, . 
\end{align*}
We show that for all 
\begin{align*}
    0 \leq n \leq \left\lfloor \frac{p}{2} \right\rfloor - 1 \, ,    
\end{align*}
the expression 
\begin{align*}
    \v_n^T \X \v_n = 0 \, . 
\end{align*}
Since $\X \succeq 0$, it follows that $\X \v_n = 0$.  

For $n=0$, the tall-tree conditions are $\b^T \e  = 1$ and $\b^T \v_1 = \tfrac{1}{2}$, hence:
\begin{align*}
    \v_0^T \X \v_0 &=  \b^T \v_1 + \v_1^T \b - (\e^T \b)^2 \\
    &= 0 \, . 
\end{align*}
We now proceed by strong induction: Let $n \leq \lfloor \tfrac{p}{2} \rfloor - 1$ be any positive integer, and assume that $\X \v_k = 0$ for all $0 \leq k \leq n - 1$. We show that $\v_n^T \X \v_n = 0$, which then completes the proof.

By hypothesis and the $p$th order tall-tree conditions, we have that 
\begin{align}\label{eq:talltreev2}
    \b^T \v_{j} = \frac{1}{(j+1)!} \qquad \textrm{for} \qquad j = 0, \ldots, 2n + 1  \, ,
\end{align}
as $2n + 1 \leq p - 1$ by the choice of $n$. 

The first step is to obtain an expression for $\R \v_n$.  Substituting the definition of $\X$ in terms of $\R$ into the induction hypothesis $\X \v_k = 0$, yields the recursion relation
\begin{equation}\label{eq:recursion}
    \R \v_{k+1} = \b \b^T \v_{k} - \A^{T} \R \v_{k} \,  \qquad \textrm{for} \qquad 0 \leq k \leq n - 1\, ,    
\end{equation}
so that $\R \v_{k+1}$ is written in terms of $\R \v_k$.  Setting $k = n - 1$ in \eqref{eq:recursion}, we can use the conditions \eqref{eq:talltreev2} on dot products $\b^T \v_k$ and iteratively eliminate $\R\v_j$ to express $\R \v_{n}$ in the basis $\{\b, \A^T \b , \ldots , (\A^T)^{n-1} \b\}$:
\begin{align}\label{eq:Rv_sum}
    \R \v_n &= \sum_{j = 0}^n \frac{(-1)^j}{(n - j)!} (\A^T)^{j} \, \b \, .
\end{align}
We then have:
\begin{alignat*}{2}
    \v_n^T \X \v_n &= \v_n^T \left(  \R\A + \A^T\R - \b \b^T  \right) \v_n \, ,  & & \\      %
      &= 2 \, \v_{n+1}^T \R \v_n  - \frac{1}{(n+1)!^2}   & \qquad &\textrm{(since } \A\v_n = \v_{n+1} \textrm{)} \, , \\
      &= 2 \, \sum_{j = 0}^n \frac{(-1)^j}{(n - j)! (n+j+2)!} - \frac{1}{(n+1)!^2}  \quad & &\textrm{(via \eqref{eq:talltreev2} and \eqref{eq:Rv_sum})} \, , \\ %      
      &=\frac{(-1)^n}{(2n+2)!} \, \underbrace{\sum_{j=0}^{2n+2}
    \begin{pmatrix}
        2n+2\\
        j
    \end{pmatrix}      
    1^{2n+2-j}(-1)^j}_{=(1-1)^{2n+2}} \, ,  & & \\
    &= 0 \, . & & 
\end{alignat*}
\end{proof} 

\begin{remark} 
Lemma~\ref{thm:nullvec} can be viewed as a generalization that algebraically stable methods admit a set of null vectors in their algebraic stability matrix $\mat{B} \mat{A} + \mat{A}^T \mat{B} - \vec{b} \vec{b}^T$, For further reference, see the proof of \cite[Lemma 13.14]{HairerWanner1996}.\myremarkend
\end{remark}

The following theorem modifies the original CSTW conditions by including the null vectors of $\X$ as additional constraints. 

\begin{theorem}[Main Result, CSTW Conditions with order conditions]\label{Thm:CSTWModified} Let $\Sf(z)$ be a $p$th order approximation to $e^z$, i.e., $(\A, \b)$ satisfy \eqref{eq:ordercondition}, and let $\M$ be any matrix whose column space is equal to the span of $[\e, \A\e, \A^2\e, \ldots, \A^{s-1} \e]$.  Then $\Sf(z)$ is $A$-stable if and only if there exists a matrix $\R\in\mathbb{S}^s$ such that
    \begin{align}\label{CSTW:AlgCondition}
        \begin{cases} 
            \R \, \e = \b \, ,& \\ 
            \X = \R \A + \A^T \R - \b \b^T \, ,& \\
            \X \A^{j-1} \e = 0 \, & \qquad \textrm{for} \quad j = 1, \, \ldots, \, \left\lfloor \frac{p}{2} \right\rfloor \, ,\\ 
            \M^T \R \M \succeq 0 \, , & \\
            \M^T \X \M \succeq 0 \, ,&
        \end{cases}
    \end{align} 
\end{theorem}
\begin{proof}
   Note that the proof of Lemma \ref{Thm:CSTWModified} holds if $\X\succeq0$ is replaced with $\M^T\X\M\succeq0$ for any matrix $\M$ whose columns span the null vectors $\A^{j-1}\e$ for $j\leq\left\lfloor\frac{p}{2}\right\rfloor$. By definition, the $\M$ appearing in Theorem \ref{Thm:CSTWModified}  contains all $\A^{j-1}\e$ for $j\leq\left\lfloor\frac{p}{2}\right\rfloor$. Therefore, any matrix pair $(\R,\X)$ satisfying the CSTW conditions \eqref{SW:AlgCondition} also satisfy \eqref{CSTW:AlgCondition}.
\end{proof}
Again, feasibility of \eqref{CSTW:AlgCondition} is a sufficient condition for $A$-stability when $\M = \I$.

An LMI for the modified CSTW conditions \eqref{CSTW:AlgCondition} can be obtained by parameterizing the $\eta_{ij}$ variables in \eqref{eq:basisR} to satisfy the additional affine constraints $\mat{X} \mat{A}^{j-1} \vec{e} = 0$ for $j=1,\ldots,\left\lfloor \frac{p}{2} \right\rfloor$. The parameterization is then substituted back into \eqref{LMI:CSTW}.

%===============================================================================
\subsection{Examples Highlighting Lemma~\ref{thm:nullvec}}\label{subsec:MainThmEx}
%===============================================================================
In this section, we provide examples highlighting the significance of Lemma~\ref{thm:nullvec}. The examples demonstrate that without incorporating the constraints imposed by the approximation property \eqref{eq:ordercondition}, computational approaches are unlikely to provide rigorous certifications for $A$-stability. 

Given a fixed pair $(\A,\b)$, the following convex set is introduced
\begin{align*}
    \mathcal{R}(\A,\b) := \big\{ \veta \in \mathbb{R}^{s(s-1)/2} \; : \;  \textrm{The LMI} \; \eqref{LMI:CSTW} \; \textrm{holds} \big\} \, ,  
\end{align*}
The (original, unmodified) CSTW conditions in Theorem~\ref{SW:1994} are equivalently rephrased as: A scheme $(\A,\b)$ is $A$-stable if and only if $\mathcal{R}(\A,\b)$ is non-empty.

It might be expected that the dimension of the convex set $\mathcal{R}$ is $s(s-1)/2$ (i.e., the dimension of symmetric matrices minus the constraints imposed by $\R\e = \b$). However, Lemma~\ref{thm:nullvec} indicates that the inequality constraint $\X \succeq 0$ (or $\M^T \X \M \succeq 0$) further reduces the dimension of $\mathcal{R}$ to be strictly smaller. Consequently, computational approaches that seek $\veta \in \mathcal{R}$ will essentially ``never'' find feasible points without correctly characterizing the low dimensional space $\mathcal{R}$ --- which is provided by Theorem~\ref{Thm:CSTWModified}.

The first example demonstrates how the zero eigenvalues of $\X$ reduce the dimension of the convex set $\mathcal{R}$.  The second example shows why $\X \succeq 0$ is a necessary hypothesis in the Lemma. 

\begin{example}[SDIRK(3,2)]\label{Example:DIRK3_2} This example constructs the set $\mathcal{R}(\A,\b)$ for the following SDIRK 3-stage $p=2$ method 
\[
\renewcommand\arraystretch{1.2}
\begin{array}{c|c}
            \vec{c} & \A\\ \hline
            & \b^T
        \end{array} = 
\begin{array}
{c|ccccc}
1 & 1 & 0 & 0 \\
\frac{3}{2} & \frac{1}{2} & 1 &0\\
1 & 1 & -1 & 1\\
\hline 
& 1 & -1 & 1
\end{array} \, . 
\]
The set $\mathcal{R}(\A,\b)$ is defined as $\veta = \left[ \eta_{12} , \, \eta_{13} , \, \eta_{23} \right]^T$ for which $\R(\veta) \succeq 0 $ and $\X(\veta) \succeq 0$ where
\begin{align*}
    \R(\veta) &= \B + \eta_{12} \, \Rb_{12} + \eta_{13}\, \Rb_{13} + \eta_{23}\, \Rb_{23} \, , \\   
    \X(\veta) &= (\B\A + \A^T\B - \b\b) + \eta_{12} \, (\Rb_{12}\A + \A^T \Rb_{12})  
    + \eta_{13}\, (\Rb_{13}\A + \A^T \Rb_{13}) \\ 
    &\quad+ \eta_{23}\, (\Rb_{23}\A + \A^T \Rb_{13}) \, ,
\end{align*}
and $\B = \mathrm{diag}( 1, \; -1, \; 1)$.  Given that $\e, \A\e$, and $\A^2 \e$ are linearly independent, the matrix $\M$ in the CSTW conditions is set to the identity matrix. 

The CSTW conditions suggest the dimension of $\mathcal{R}(\A,\b)$ is $3$. However, the implication of Theorem~\ref{thm:nullvec} is that the dimension is in fact $1$. Since $p = 2$, Theorem~\ref{thm:nullvec} allows for the inclusion of the constraint $\X \e = 0$ within $\mathcal{R}$, i.e.,
\begin{align*}
    \mathcal{R}(\A,\b) = \big\{ \veta \in \mathbb{R}^3 \; : \; 
    \R(\veta) \succeq 0 \, , \; \X(\veta ) \succeq 0 \;, \X \e = 0 \big\} \, . 
\end{align*}
The constraint $\X \e = 0$ imposes two independent linear equations on $\veta$ whose solution forces $\eta_{12} = 3$ and $\eta_{32} = 2$. Thus,
\begin{align*}
    \mathcal{R}(\A,\b) &=  \big\{ \veta \in \mathbb{R}^3 \; : \; \eta_{12} = 3 \, ,  \; \; \eta_{32} = 2 \, , \; \; 
    \R(\eta_{13}) \succeq 0 \, , \; \X(\eta_{13} ) \succeq 0  \big\} \, ,
\end{align*}
where
\begin{align*}
    \R(\eta_{13} ) &= \begin{bmatrix}
    \phantom{-}4 & -3 & \phantom{-}0 \\
    -3 & \phantom{-}4 & -2 \\
    \phantom{-}0 & -2 & \phantom{-}3
\end{bmatrix} + \eta_{13} \begin{bmatrix}
        \phantom{-}1 & \phantom{-}0 & -1 \\
        \phantom{-}0 & \phantom{-}0 & \phantom{-}0 \\
        -1 & \phantom{-}0 & \phantom{-}1 
    \end{bmatrix} \, , \\
    \X(\eta_{13} ) &= \begin{bmatrix}
        \phantom{-}4 & -5 & \phantom{-}1 \\
        -5 & \phantom{-}11 & -6 \\
        \phantom{-}1 & -6 & \phantom{-}5
    \end{bmatrix} + \eta_{13} 
    \begin{bmatrix}
        \phantom{-}0 & \phantom{-}1 & - 1\\
        \phantom{-}1 & \phantom{-}0 & -1\\
        -1 & -1 & \phantom{-}2
    \end{bmatrix} \, .  
\end{align*}
Figure~\ref{fig:my_label} visualizes non-empty $1$-dimensional set $\mathcal{R}$, which notably has empty interior.

\begin{figure}[htbp]
\centering
\includegraphics[width=.75\linewidth]{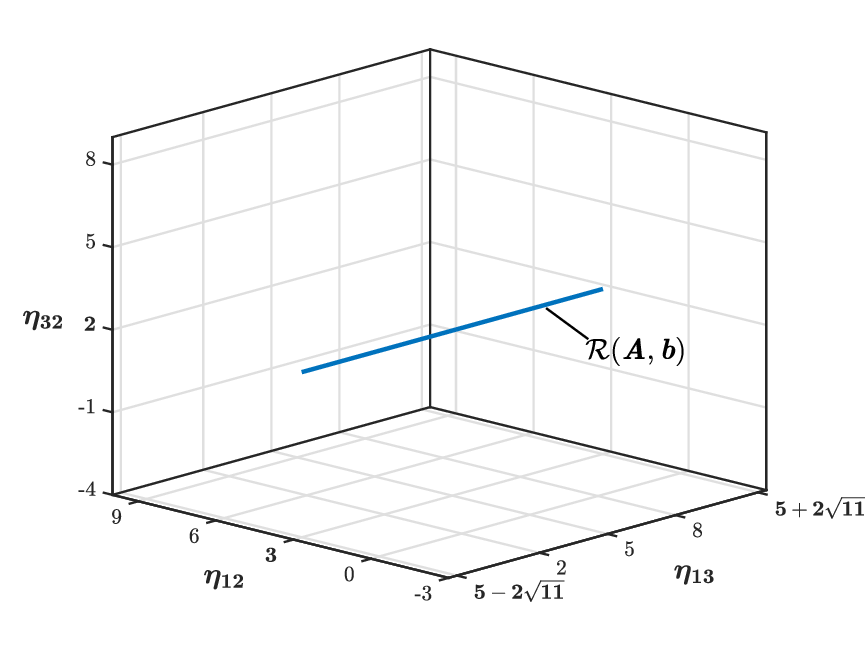}
\caption{The blue line is a visualization of $\mathcal{R}(\A,\b)$ for SDIRK(3,2) in Example~\ref{Example:DIRK3_2}. Note the set $\mathcal{R}$ has dimension $1$, as characterized by Theorem~\ref{Thm:CSTWModified}, which is lower than the expected $3$ dimensional set suggested by Theorem~\ref{SW:1994}. }
\label{fig:my_label}
\end{figure}
\end{example}

\begin{example}[Hammer \& Hollingsworth]\label{Example:HH_2_4}  This example demonstrates why $\X \succeq 0$ is required as a hypothesis in Theorem~\ref{thm:nullvec} for $\X$ to have zero eigenvalues. The Hammer \& Hollingsworth two stage $p = 4$ method \cite[Table II.7.3]{HairerNorsettWanner1993} is represented as
\[
\renewcommand\arraystretch{1.2}
\begin{array}{c|c}
            \vec{c} & \A\\ \hline
            & \b^T
        \end{array} = 
\begin{array}
{c|cc}
  \tfrac{1}{2} - \tfrac{\sqrt{3}}{6}  & \tfrac{1}{4} & \tfrac{1}{4}-\tfrac{\sqrt{3}}{6} \\
 \tfrac{1}{2}+\tfrac{\sqrt{3}}{6} & \tfrac{1}{4}+\tfrac{\sqrt{3}}{6} & \tfrac{1}{4} \\
\hline 
 & \tfrac{1}{2} & \tfrac{1}{2}
\end{array} \, . 
\]
The scheme yields the following stability function and $E$-polynomial
\begin{align*}
    \Sf(z) = \frac{1 + z/2 + z^2/12}{1 - z/2 + z^2/12} \,\qquad \textrm{and} \qquad  \quad E(y) = 0  \, .
\end{align*}
Since $E(y) \geq 0$, the scheme is $A$-stable and $\mathcal{R}(\A,\b)$ is not empty. 

The set $\mathcal{R}(\A,\b)$ is characterized by $\eta \in \mathbb{R}$ via  
\begin{align*}    
    \R(\eta) &= \frac{1}{2} \I + \eta \begin{pmatrix}
        \phantom{-}1 \\
        -1
    \end{pmatrix}\ 
    \begin{pmatrix}
        1 & -1
    \end{pmatrix} \,,  \qquad 
     \X(\eta) = -\frac{\sqrt{3}}{2} \eta \begin{pmatrix}
         1 & 0 \\
         0 & -1
     \end{pmatrix} \, . 
\end{align*}
The only value of $\eta$ for which $\X(\eta) \succeq 0, \R(\eta) \succeq 0$ is $\eta = 0$, i.e., $\mathcal{R}(\A,\b) = \{0\}$. 

Note that $\X \succeq 0$ is required for $\X$ to have a non-trivial null space.  For any value of $\eta \neq 0$, the matrix $\X$ is invertible and has no null vectors, even though $(\A,\b)$ satisfies the order conditions.  When $\X \succeq 0$, it has two linear independent null vectors predicted by the Theorem~\ref{thm:nullvec}: $\{ \e, \A \e\}$ --- and thus must be the zero matrix (i.e., $\eta = 0$). 
\end{example}

%==================================================================================================
\section{Examples and Results for Runge-Kutta Schemes}\label{sec:Examples}
%==================================================================================================
To provide examples of using LMIs to certify the stability of RK methods, we start with an overview in Subsection~\ref{subsec:numericalDetails} that details the general numerical approach. We then apply the approach to an idealized example, the $A$-stable SDIRK(5,4) scheme, in Subsection~\ref{ex:st_dirk4}. Additionally, we examine $A$-stability for several recently devised schemes developed via numerical software \cite{AlamriKetcheson2024, BiswasKetchesonSeiboldShirokoff2023}. These schemes do not satisfy the tall-tree order conditions exactly but admit a residual of typical size $\mathcal{O}(10^{-15})$. The section concludes with examples establishing rigorous bounds on $\alpha$ for $A(\alpha)$-stability.

%==================================================================================================
\subsection{Computational Details for Rigorous Certification}\label{subsec:numericalDetails}
%==================================================================================================
We describe computational details for rigorously verifying stability via feasibility of an LMI, proceeding in two steps. 

First, we use CVX, a package for specifying and solving convex programs \cite{cvx, gb08}, to numerically solve one of the following semi-definite programming problems:
\begin{itemize}
    \item To demonstrate the non-negativity of the $E$-polynomial from \S~\ref{subsec:EpolyLMI}, i.e., assess the feasibility of \eqref{Eq:SOS_LMI}, we solve
\begin{opt}\label{opt:SDP_Epoly}
    \textrm{Minimize: }&  1 \\
    \textrm{Subject to: }&  \F(\vec{\eta}) = \P + \sum_{j=1}^d \eta_j \, \N_j \succeq 0 \, ,                  
\end{opt}
\item To assess the feasibility of the modified CSTW approach \eqref{CSTW:AlgCondition}, we solve
\begin{opt}\label{opt:SDP_CSTW}
    \textrm{Minimize: }&  1 \\
    \textrm{Subject to: }&  \R \, \e = \b \, , \\ 
                         & \X = \R \A + \A^T \R - \b \b^T \, , \\
                         & \X \A^{j-1} \e = 0 \, ,   \qquad\qquad\qquad \textrm{for} \quad j = 1, \, \ldots, \, \left\lfloor \frac{p}{2} \right\rfloor \, ,\\
                         & \M^T \R \M \succeq 0 \, , \\
                         & \M^T \X \M \succeq 0 \, , 
\end{opt}where $\M$ is any matrix containing the columns where $r$ is the largest integer for which $\A^{r}\e$ can be written as a linear combination of $[\vec{e} , \; \A \e , \; \ldots, \A^{r-1} \e]$.
From an implementation standpoint, we convert the constraints in \eqref{opt:SDP_CSTW} (as discussed in \S~\ref{subsec:Convex_Opt}) into an LMI. 
\end{itemize}
While not a proof, a positive output of CVX provides numerical evidence that the convex set in question is feasible. 

Second, we use the output of CVX to construct a rigorous certificate of stability in settings where $\mat{P} \in \mathbb{Q}^{m \times m}$, or $\A \in \mathbb{Q}^{s \times s}, \b \in \mathbb{Q}^s$.  We convert the double-precision floating-point outputs of CVX into symbolic rational entries and perform exact symbolic LDL factorizations for $\F$ in \eqref{opt:SDP_Epoly} or $\R, \X$ in \eqref{opt:SDP_CSTW} to yield matrices $\L, \D$ with rational entries. A rigorous certificate is then ensured provided $\D \succeq 0$. 

For example, in the case of \eqref{opt:SDP_Epoly}, CVX yields a set of $\eta_j^* \in \mathbb{Q}$ for $(1 \leq j \leq d)$; substituting $\{\eta_j^*\}_{j=1}^d$ back into the LMI and symbolically computing an LDL factorization yields
\begin{align*}
    \F(\vec{\eta}^*) = \L \D \L^T   \qquad \textrm{with} \qquad \L, \D \in \mathbb{Q}^{m \times m} \, .
\end{align*}
For presentation and simplicity, we often round the output values of CVX $\eta_j^*$ to nearby $\bar{\eta}_j \in \mathbb{Q}$ with integer numerator and denominators having fewer digits, yet still yielding positive certificates of feasibility.

\begin{remark}
    Throughout this section, we present coefficient matrices wherever possible. However, in several of the practical examples, the coefficients of various matrices, e.g., $\P$, $\X$, etc., as well as the LDL factorizations of $\F$ or $\X, \R$ admit rational values exceeding 100 digits. \myremarkend
\end{remark}

\begin{remark}
    Note that if the coefficients of $\P$ or $\A,\b$ lie in a field extension $\mathbb{F}$ of $\mathbb{Q}$, the LDL factorization also admits matrices $\L, \D$ in $\mathbb{F}$. Therefore, this approach generalizes in a straightforward way as long as one can determine the sign of any element $x \in \mathbb{F}$, such as with interval arithmetic.  In addition, it is worth noting that schemes $\A, \b$ with irrational entries may have $E$-polynomials with rational entries and coefficient matrices $\P$, allowing this approach to apply without any further modification (e.g., the scheme \eqref{RK_ramos_vigo2007} has coefficients in $\mathbb{Q}[\sqrt{2}]$ but $E$-polynomial with coefficients in $\mathbb{Q}$) \myremarkend.
\end{remark}

%==================================================================================================
\subsection{Example of $A$-stability in SDIRK(5,4)}\label{ex:st_dirk4}
%==================================================================================================
The first example demonstrates the hybrid computational-analytic LMI solution approach to verify $A$-stability for an idealized example that satisfies the tall tree order conditions exactly and is known to be $A$-stable. The scheme is a Diagonally Implicit RK method, represented by the following Butcher tableau \cite[Table~6.5, Chapter IV.6]{HairerWanner1996}.
\begin{align}\label{DIRK5_4}
\renewcommand\arraystretch{1.2}
\begin{array}{c|c}
            \vec{c} & \A \\ \hline
            & \b^T
        \end{array}\;
        =
\begin{array}
{c|ccccc}
\frac{1}{2} & \frac{1}{4} & & & &\\
\frac{3}{4} & \frac{1}{2} & \frac{1}{4} & & &\\
\frac{11}{20} & \frac{17}{50} & -\frac{1}{25} & \frac{1}{4} & &\\
\frac{1}{2} & \frac{371}{1360} & -\frac{137}{2720} & \frac{15}{544} & \frac{1}{4} &\\
1 & \frac{25}{24} & -\frac{49}{48} & \frac{125}{16} & -\frac{85}{12} & \frac{1}{4} \\
\hline 
& \frac{25}{24} & -\frac{49}{48} & \frac{125}{16} & -\frac{85}{12} & \frac{1}{4}
\end{array} \, .
\end{align}
We provide two proofs: one using the $E$-polynomial LMI and a second based on modified CSTW LMI conditions. 

%================================================
\subsubsection{E-ploynomial linear matrix inequality for SDIRK(5,4)}
%================================================
The SDIRK(5,4) scheme in \eqref{DIRK5_4} has an $E$-polynomial $E(y)=y^6(9y^{4}-64y^2+512)$ ($\alpha = \frac{\pi}{2}$), which after factoring out the largest monomial factor, yields
\begin{align*}
        F(y) := 9y^{4}-64y^2+512 = \vec{y}^T \, 
            \F(\eta) \, \vec{y} \, , 
\end{align*}
where 
\begin{align}\label{Eq:DIRK54_F}
    \F(\eta) = 
        \begin{bmatrix}
                512 & 0 & 0\\
                0&  -64 & 0\\
                0& 0 & 9  
            \end{bmatrix}
            +
            \eta \begin{bmatrix}
               0 & 0 & 1 \\
               0 & -2  & 0\\
                1 & 0 & 0
            \end{bmatrix} \, .
\end{align}

For this example, the space $\mathcal{N}_3$ (introduced in \S~\ref{subsec:BackgrounPolyLMI}) has dimension $d = 1$ and is spanned by the second matrix in \eqref{Eq:DIRK54_F}. The numerical solution \eqref{opt:SDP_Epoly} for $\F(\eta)$ yields the following output:
\begin{align*} 
    \eta^* = -61.786375823904734 \, . 
\end{align*}
The fact that CVX obtains a solution is numerical evidence suggesting $E(y) \geq 0$, indicating that the scheme is $A$-stable.

We can turn this numerical evidence into a rigorous proof by substituting rational values of $\eta$ close to this computational output and factorizing exactly $\F(\eta) = \L\D\L^T$ with $\L, \D \in \mathbb{Q}^{3\times 3}$. While there is no requirement to substitute integer values of $\eta$, for the sake of presentation, we round the CVX output $\eta^*$ to $\overline{\eta} = -60$ and obtain:
\begin{align*}
    \F(-60) = \begin{bmatrix}
        512 & 0 & -60 \\
           0 & 56 & 0 \\
            -60 & 0 & 9
    \end{bmatrix} = \L \, \D \, \L^T \, , 
\end{align*}
where 
\begin{align*}    
    \L=
    \begin{bmatrix}
            1 & 0 & 0\\[2pt]
            0 & 1 & 0\\[2pt]
            \text{-}\frac{15}{128} & 0 & 1
    \end{bmatrix} \, ,
    \qquad \textrm{and} \qquad 
    \D=
    \begin{bmatrix}
        512 & 0 & 0\\[2pt]
        0 & 56 & 0\\[2pt]
        0 & 0 & \frac{63}{32}
    \end{bmatrix} \succeq 0 \, . 
\end{align*}
This results in the following SOS representation of $E(y)$, certifying $A$-stability:
\begin{align*}
    E(y) = y^6 \, \left( 512\left(\frac{15}{128}y^2-1\right)^2 + 56y^2+ \frac{63}{32}y^4 \right) \, . 
\end{align*}

\begin{remark}
    Since the linear constraints are exactly parameterized, there is robustness in the choice of $\eta$. For example, the value of $\bar{\eta} = -32$ similarly yields a certificate of $A$-stability since $\F(-32) =  \L \, \D \, \L^T \succeq 0$, where
\begin{align*}    
    \L &= \left[\begin{array}{ccc} 1 & 0 & 0\\ 0 & 1 & 0\\ \text{-}\frac{1}{16} & 0 & 1 \end{array}\right] \, , \qquad 
    \textrm{and} \qquad  
    \D = \left[\begin{array}{ccc} 512 & 0 & 0\\ 0 & 0 & 0\\ 0 & 0 & 7 \end{array}\right] \succeq 0 \, . 
\end{align*}
    \myremarkend
\end{remark}

%================================================
\subsubsection{The Modified CSTW-Method for SDIRK(5,4)}
%================================================
Using the SDIRK(5,4) scheme in \eqref{DIRK5_4} as an example, this section highlights several key differences in the modified CSTW approach relative to the $E$-polynomial.

Since $p = 4$, Theorem \ref{thm:nullvec} adds two linear constraints on $\X$ (corresponding to null vectors $\e$ and $\A\e$) into the associated LMI. Consequently, problem \eqref{opt:SDP_CSTW} has three degrees of freedom, denoted by the vector $\vec{\eta}$. The rounded numerical solution of \eqref{opt:SDP_CSTW} yields matrices

\[
\renewcommand\arraystretch{1.5}
\setlength\arraycolsep{2pt}
\begin{aligned}
\X^*&=\left[\begin{array}{ccccc} \frac{729823}{97920} & \text{-}\frac{348733}{195840} & \frac{875727}{21760} & \text{-}\frac{334871}{7200} & \frac{237}{400}\\ \text{-}\frac{348733}{195840} & \frac{170083}{391680} & \text{-}\frac{1259867}{130560} & \frac{160397}{14400} & \text{-}\frac{57}{400}\\ \frac{875727}{21760} & \text{-}\frac{1259867}{130560} & \frac{5678645}{26112} & \text{-}\frac{241217}{960} & \frac{16}{5}\\ \text{-}\frac{334871}{7200} & \frac{160397}{14400} & \text{-}\frac{241217}{960} & \frac{1045211}{3600} & \text{-}\frac{1479}{400}\\ \frac{237}{400} & \text{-}\frac{57}{400} & \frac{16}{5} & \text{-}\frac{1479}{400} & \frac{19}{400} \end{array}\right]
\\
\end{aligned}
\quad \textrm{and}\quad
\begin{aligned}
\R^*&=\left[\begin{array}{ccccc} \frac{195061}{16320} & \text{-}\frac{42157}{10880} & \frac{416905}{6528} & -72 & \frac{11}{10}\\ \text{-}\frac{42157}{10880} & \frac{131641}{65280} & \text{-}\frac{324335}{13056} & \frac{1259}{48} & \text{-}\frac{11}{20}\\ \frac{416905}{6528} & \text{-}\frac{324335}{13056} & \frac{4888637}{13056} & \text{-}\frac{99107}{240} & \frac{73}{10}\\ -72 & \frac{1259}{48} & \text{-}\frac{99107}{240} & \frac{34459}{75} & \text{-}\frac{391}{50}\\ \frac{11}{10} & \text{-}\frac{11}{20} & \frac{73}{10} & \text{-}\frac{391}{50} & \frac{11}{50} \end{array}\right],
\end{aligned}
\]
which then admits LDL factorizations of the form 
\begin{align*}
    \X^* =\L_X\D_X\L_X^T\,, \qquad\qquad 
    \R^* = \L_R\D_R\L_R^T \, ,
\end{align*}
where the matrices have coefficients in $\mathbb{Q}$ with 
\[
\renewcommand\arraystretch{1.5}
\setlength\arraycolsep{2pt}
\begin{aligned}
\D_X=\operatorname{diag}\left[\begin{array}{c}
\frac{729823}{97920} \\
\frac{2466451}{280252032} \\
\frac{7352143}{246645100} \\
0  \\
0
\end{array}\right]\,,
\qquad 
\D_R=\operatorname{diag}\left[\begin{array}{c} 
\frac{195061}{16320}\\
\frac{28479739}{37451712} \\
\frac{3647946461}{341756868} \\
\frac{3800443925}{43775357532}\\
\frac{103805}{2104052} 
\end{array}\right] \, .
\end{aligned}
\]

The fact that $\D_X$ admits two zero eigenvalues is by construction and follows from Theorem~\ref{Thm:CSTWModified}. The certification of $A$-stability is thus confirmed as the pair $\X^*, \R^*$ satisfy the CSTW condition in exact arithmetic.

%==================================================================================================
\subsection{Schemes Failing to Satisfy The Tall-Tree Order Conditions}
%==================================================================================================
Building on the idealized example from the previous subsection, we now focus on certifying stability for RK schemes developed from numerical solutions of the order conditions (i.e., using numerical optimization software). The schemes considered here have coefficients that do not exactly satisfy the order conditions \eqref{eq:ordercondition}. Instead, the coefficients satisfy these conditions with a small residual. Failing to satisfy the tall-tree order conditions leads to two complications: the $E$-polynomial no longer satisfies \eqref{EPoly_Asymptotics} but instead includes low-degree monomials with small coefficients; additionally, the matrix $\X$, used in the CSTW conditions, no longer admits exact zero eigenvalues, but rather small nonzero ones. 

We test schemes from two sources:
\begin{itemize}
    \item Diagonally implicit Runge-Kutta schemes with weak stage order \cite{BiswasKetchesonSeiboldShirokoff2023} (cf. \cite{BiswasKetchesonSeiboldShirokoff2024, BiswasKetchesonRobertsSeiboldShirokoff2024}). These schemes were developed to alleviate the effects of \emph{order reduction} on stiff problems, primarily arising from spatial discretizations of linear partial differential equations. The schemes are denoted as WSO DIRK(s,p,q), where $s$, $p$, and $q$ denote the number of stages, classical order, and weak stage order, respectively.
    \item (Very high order) Diagonally implicit Runge-Kutta schemes with additional practical properties, developed in \cite{AlamriKetcheson2024}. 
\end{itemize}
It is worth noting that while both \cite{BiswasKetchesonSeiboldShirokoff2023, AlamriKetcheson2024} provide strong numerical evidence for $A$-stability (e.g., a solution with small residual was provided for the CSTW LMI in \cite{AlamriKetcheson2024}) neither work provides a rigorous certificate in the form of an exact solution to one of the LMI's. 

We adopt two strategies for testing $A$-stability, outlined as follows:

\medskip
\noindent
{\bf Strategy 1:}
The Butcher coefficients $(\A, \b)$ are reported as decimal expansions, typically to 16 digits of accuracy (e.g., the accuracy of double-precision floating-point arithmetic from which they were obtained). Treating the coefficients as rational values, we test for the non-negativity of the associated $E$-polynomial, which also has rational coefficients. This strategy determines whether the dynamics \eqref{eq:RK_methods} are $A$-stable despite the coefficients not exactly satisfying the tall-tree order conditions.  

In contrast to the $E$-polynomial approach, challenges arise when using the CSTW  approach to obtain a rigorous certificate of $A$-stability. The modified CSTW conditions depend on the exact satisfaction of the tall tree order conditions. However, in cases where the Butcher coefficients approximate these conditions with a small residual, the null vectors of $\mat{X}$ are no longer applicable constraints in the SDP, and the original CSTW conditions must be used. The feasible set of the CSTW conditions may then take the form of a tubular domain of a low-dimensional set, thereby introducing computational challenges to the numerical solution.

\medskip
\noindent
{\bf Strategy 2:} To overcome the challenges in Strategy 1 due to approximations in the tall tree order conditions, we assess the $A$-stability of perturbed schemes $(\mat{\tilde{A}}, \vec{\tilde{b}})$ that simultaneously:
\begin{itemize}
    \item Have coefficients in $\mathbb{Q}$ and satisfy the tall-tree order conditions \eqref{eq:ordercondition} in exact arithmetic; 
    \item Are perturbations of the reported scheme $(\A, \b)$ in the literature, satisfying an error bound
    \begin{align}\label{Eq:PertSchemeErrorEst}
        |b_i - \tilde{b}_i| < \epsilon_b \quad \textrm{and} \quad
        |a_{ij} - {\tilde{a}}_{ij} | < \epsilon_A \quad \textrm{for} \quad i,j=1,\ldots,s \, . 
    \end{align}    
\end{itemize}
Since the perturbed schemes, $(\mat{\tilde{A}}, \vec{\tilde{b}})$, satisfy the tall-tree order conditions, their $E$-polynomials satisfy \eqref{EPoly_Asymptotics} and the associated $\X$ matrix in the CSTW approach admits null vectors (in exact arithmetic) characterized by Theorem~\ref{thm:nullvec}. Thus, both the $E$-polynomial and CSTW approaches provide pathways for certifying $A$-stability for $(\mat{\tilde{A}}, \vec{\tilde{b}})$. 

Several schemes that fail to achieve rigorous certificates of $A$-stability under Strategy 1 are shown to be $(\epsilon_A,\epsilon_b)$-close to schemes that attain rigorous $A$-stability certificates under Strategy 2.

%================================================
\subsubsection{Certification of $A$-stability via Strategy 1}
%================================================
Here, we verify $A$-stability of three schemes using Strategy 1.  Since the schemes do not satisfy the order conditions exactly, the $E$-polynomial ($\alpha = \frac{\pi}{2}$) has the form:
\begin{align*}
    {E}(y)=y^2\vec{y}^T \mat{F}(\eta) \vec{y}, \qquad \textrm{with} \qquad \mat{F}(\eta) = \mat{P} + \sum_{j=1}^{d} \eta_{j} \Nb_j \,, 
\end{align*}
where $\mat{P}$ and $\mat{N}_j$ are as in \S~\ref{subsec:BackgrounPolyLMI}. The details for each scheme are as follows:

\begin{enumerate}
    \item WSO DIRK(12,5,4) developed in \cite{BiswasKetchesonSeiboldShirokoff2023}: A symbolic computation of $E(y)$ yields a diagonal matrix $\P$, containing the coefficients of $E(y)$. Since $\P \succeq 0$, it follows that $\mat{F}(\eta) \succeq 0$ trivially when $\veta = 0$. Consequently, no SDP is required, as the LMI is immediately satisfied.
    \item WSO DIRK(7,4,4), also developed in \cite{BiswasKetchesonSeiboldShirokoff2023}: A symbolic computation of ${E}(y)=y^2\vec{y}^T\P\vec{y}$, yields a diagonal matrix $\P \in \mathbb{Q}^{7 \times 7}$ with two negative coefficients
        \begin{align*}
            p_6 < 0\,, \qquad p_8 <  0 \, .
        \end{align*} 
        Since $\mat{F}(0)$ does not satisfy the LMI, we seek solutions to the LMI (which has dimension $d = 15$) via an SDP. The $E$-polynomial SDP identifies a candidate feasible solution $\vec{\eta^*}$, which for computational simplicity, we round to $\vec{\bar{\eta}}$, resulting in the matrix:
        \begin{align*}
            \F(\vec{\bar{\eta}})=\P-108420\Nb_{10}+20\Nb_{12}-3420\Nb_{13}-30\Nb_{15}=\L\D\L^T , 
        \end{align*}
    with $\L, \D \in \mathbb{Q}^{7 \times 7}$ and $\D\succ0$.
    \item DIRK(13,8)[1]A[(14,6)A] developed in \cite{AlamriKetcheson2024}: Similar to the previous case, symbolic computation of the $E$-polynomial yields $\P \in \mathbb{Q}^{13\times 13}$ with two negative coefficients
        \begin{align*}
            p_{14} < 0\,, \qquad p_{22} < 0 \, , 
        \end{align*}
        requiring an SDP to find a potential sum of squares representation for $E(y)$. The $E$-polynomial LMI has dimension $d = 66$, and the corresponding SDP identifies a feasible solution $\vec{\eta^*}$, which upon rounding yields $\vec{\bar{\eta}}$ (reported in Appendix~\ref{Appendix:DIRK13_8}), resulting in the matrix $\F(\vec{\bar{\eta}})=\L\D\L^T$, where the diagonal matrix $\D \succ 0$. 
\end{enumerate}
The complete set of coefficients for $\vec{\bar{\eta}}, \mat{L}, \mat{D}$, as well as generating Matlab code, can be found in the supplemental material.

The results are formalized with the following Lemma.
\begin{lemma}
    The $E$-polynomials for the weak stage order schemes \textup{DIRK(12,5,4)} and \textup{DIRK(7,4,4)} in \cite{BiswasKetchesonSeiboldShirokoff2023}, and \textup{DIRK(13,8)[1]A[(14,6)A]} in \cite{AlamriKetcheson2024} are non-negative, and the associated schemes are $A$-stable. 
\end{lemma}

%================================================
\subsubsection{Certification of $A$-stability via Strategy 2}\label{subsubsection:PertSchemes}
%================================================
We now apply Strategy 2 to certify $A$-stability for schemes where Strategy 1 fails due to the $E$-polynomial for $(\mat{A}, \vec{b})$ being negative near the origin. 

Throughout, tildes denote quantities for the perturbed scheme $(\mat{\tilde{A}}, \vec{\tilde{b}})$, e.g., $\mat{\tilde{P}}$ is the matrix containing the coefficients of the $E$-polynomial $\tilde{E}(y)$. 

\begin{enumerate}
    \item DIRK(6,6)[1]A[(7,5)A] developed in \cite{AlamriKetcheson2024}: The original scheme's E-polynomial, ${E}(y)=y^2\vec{y}^T\P\vec{y}$, possesses three negative coefficients in $\mat{P}$
    \begin{align*}
        p_{0} < 0, \qquad p_2 < 0, \qquad p_4 < 0 \, .
    \end{align*}
    Given that $p_0 < 0$, $E(y)$ is negative for values of $y$ near the origin, indicating that the scheme is not $A$-stable. This inability to achieve $A$-stability directly follows from the failure to satisfy the tall-tree order conditions --- which, if satisfied, would ensure $p_0 = p_2 = p_4 = 0$.

    We introduce the perturbed scheme
    \[
        \renewcommand\arraystretch{1.5}
        \begin{aligned}
            \mat{\tilde{A}}&=\left[\begin{array}{cccccc}
            \frac{33128226}{109158329} & & & & &\\
            \text{-}\frac{254432096}{909477001} & \frac{51289103}{102571593} & & & &\\
            \frac{33289838}{118645151} & \text{-}\frac{825218320}{1881654059} & \frac{130993323}{602959172} & & &\\
            \text{-}\frac{13583292}{200438515} & \frac{156154430}{158643099} & \text{-}\frac{65409371}{245235917} & \frac{81765600}{330141853} & &\\
            \frac{16354062}{130133299} & \text{-}\frac{247816720}{248961507} & \frac{169383005}{222482121} & \text{-}\frac{49241043}{234166886} & \frac{79900588}{92184791} &\\
            \text{-}\frac{34719176}{94331171} & \text{-}\frac{155737141}{155748342} & \frac{42945649}{80312134} & \text{-}\frac{50573402}{289227347} & \frac{678237381}{1102812170} & \frac{31879369}{45767530}
            \end{array}\right],
        \end{aligned}
    \]
    where the coefficient vector $\vec{\tilde{b}} \in \mathbb{Q}^{6}$ are chosen to exactly satisfy the tall-tree conditions \eqref{eq:ordercondition} with $p = 6$. The pair $(\mat{\tilde{A}}, \vec{\tilde{b}})$ satisfy the error bound \eqref{Eq:PertSchemeErrorEst} with $\epsilon_A = 5\cdot10^{-17}$, and $\epsilon_b = 6\cdot10^{-15}$.

    $A$-stability for the perturbed scheme is certified via two approaches:
    \begin{itemize}
        \item The $E$-polynomial $\tilde{E}(y) = y^{8} \vec{y}^T \mat{\tilde{P}} \vec{y} \geq 0$ follows immediately since  $\mat{\tilde{P}} \in \mathbb{Q}^{3\times 3}$ satisfies $\mat{\tilde{P}} \succeq 0$.
        \item The modified CSTW approach \eqref{opt:SDP_CSTW} yields an LMI in 3 variables. The numerical solution produces a solution pair:
        \begin{align*}
            \tilde{\X}^*= \tilde{\L}_X \tilde{\D}_X \tilde{\L}^T_X 
            \quad \textrm{and}\quad 
            \tilde{\R}^*= \tilde{\L}_R \tilde{\D}_R \tilde{\L}^T_R \, ,
        \end{align*}
        with $\tilde{\D}_X \succeq 0$ ($\tilde{\D}_X$ has $4$ zero eigenvalues by Theorem~\ref{Thm:CSTWModified}), and $\tilde{\D}_R\succeq 0$. 
    \end{itemize}    
    \item SDIRK(9,6)[1]SAL[(9,5)A] developed in \cite{AlamriKetcheson2024}: Similar to the previous example, this scheme is not $A$-stable as the $E$-polynomial is negative near the origin. This is reflected in the matrix $\P \in \mathbb{Q}^{9 \times 9}$ having negative coefficients
    \begin{align*}
        p_{0} < 0, \qquad p_2 < 0, \qquad p_4 < 0 \, ,
    \end{align*}
     a consequence of failing to satisfy the tall-tree order conditions.
    A perturbed scheme $(\mat{\tilde{A}}, \vec{\tilde{b}})$, defined in Appendix~\ref{Appendix:PertSchemes}, satisfies the error bound \eqref{Eq:PertSchemeErrorEst} with $\epsilon_A = \epsilon_b = 8\cdot10^{-15}$.

    Similar to DIRK(6,6), $A$-stability for the perturbed scheme $(\mat{\tilde{A}}, \vec{\tilde{b}})$ is certified through two approaches:
    \begin{itemize}
        \item The $E$-polynomial $\tilde{E}(y) = y^{8} \vec{y}^T \tilde{\P} \vec{y} \geq 0$, with $\tilde{\P} \in \mathbb{Q}^{6 \times 6}$, is immediately non-negative since \(\tilde{\P} \succeq 0\), avoiding the need for solving an SDP. 
        \item The modified CSTW approach, which utilizes an LMI with $15$-degree-of-freedom, is solved via SDP. The numerical solver identifies optimal pairs \((\tilde{\R}^*, \tilde{\X}^*)\), yielding exact LDL factorizations over \(\mathbb{Q}\), with all matrices \(\D\) being positive semidefinite. 
    \end{itemize}

    \item WSO DIRK(12,5,5) developed in \cite{BiswasKetchesonSeiboldShirokoff2023}: Similar to the previous two examples, this scheme also fails to be $A$-stable due to the lowest order terms in the $E$-polynomial having negative coefficients
    \begin{align*}
        p_0 < 0,\qquad p_2 < 0,
    \end{align*} 
    implying that $E(y) < 0$ for small $y$. 
    A perturbed scheme $(\mat{\tilde{A}}, \vec{\tilde{b}})$ is defined in Appendix~\ref{Appendix:PertSchemes} and satisfies \eqref{Eq:PertSchemeErrorEst} with $\epsilon_A =\epsilon_b= 9\cdot10^{-15}$.

    For $(\mat{\tilde{A}},\mat{\tilde{b}})$, the $E$-polynomial has the form $\tilde{E}(y)=y^6\vec{y}^T\mat{\tilde{F}}(\veta)\vec{y}$, where $\mat{\tilde{P}} \in \mathbb{Q}^{10 \times 10}$ is not positive definite. The associated LMI for the non-negativity of $\mat{\tilde{F}}$ contains $\dim \mathcal{N}_{10} = 36$ degrees of freedom. The Appendix~\ref{Appendix:PertSchemes} presents a solution $\vec{\bar{\eta}}$ that yields $\mat{\tilde{F}}(\vec{\bar{\eta}}) =\mat{\tilde{L}} \mat{\tilde{D}} \mat{\tilde{L}}^T$, with $\mat{\tilde{D}} \succeq 0$, thereby certifying $A$-stability.  
\end{enumerate}

 We formalize the result in the following Lemma.

\begin{lemma} There exist perturbed schemes $(\mat{\tilde{A}}, \vec{\tilde{b}})$ for \textup{DIRK(6,6)[1]A[(7,5)A]} and \newline \textup{SDIRK(9,6)[1]SAL[(9,5)A]} from \cite{AlamriKetcheson2024}, and \textup{DIRK(12,5,5)} from \cite{BiswasKetchesonSeiboldShirokoff2023}, which 
satisfy the error bound \eqref{Eq:PertSchemeErrorEst} with $\epsilon_A, \epsilon_b = \mathcal{O}(10^{-15})$. These schemes simultaneously satisfy the tall-tree conditions and are $A$-stable. 
\end{lemma}
%==================================================================================================
\subsection{Examples Establishing \(A(\alpha)\) Stability}\label{subsec:RamosVigo}
%==================================================================================================
We now shift attention and establish rigorous bounds on \(\alpha\), for $A(\alpha)$-stability of two schemes that are not \(A\)-stable. 

First note that \(E(y^2; \beta)\), where \(\beta := \cos(\alpha)\), is a bivariate polynomial in \(y, \beta\) with rational coefficients. An initial rational value of \(\beta\) is fixed, and the associated \(E\)-polynomial SDP \eqref{opt:SDP_Epoly} is solved to determine whether the associated feasible set is nonempty. The value of \(\beta\) is incrementally decreased until \eqref{opt:SDP_Epoly} indicates that the feasible set is empty. The last \(\beta\) value before the feasible set is reported non-empty is \(\beta^*\). The corresponding $\alpha^*=\cos^{-1}(\beta^*)$ is the lower bound for the maximal angle \(\alpha\) at which the scheme is confirmed to be \(A(\alpha)\)-stable. An exact rational certificate in the form of an LDL factorization is then reported.

The result is formalized as follows:
\begin{lemma}
    The schemes \eqref{RK_ramos_vigo2007} and \eqref{eq:skvortsov} are $A(\alpha)$-stable for some maximum angle $\alpha$ that is bounded below by the angle $\alpha^*$, where $\alpha^* = \cos^{-1} \beta^*$ and $\beta^*$ is defined in \eqref{Eq:VR_BetaOpt} and \eqref{Eq:BetaOptSkv} respectively.
\end{lemma}

%==================================================================================================
\subsubsection{The IRK(4,4) Scheme of Ramos and Vigo}
%==================================================================================================

The following scheme, developed by Ramos and Vigo, is a $4$-stage, $4$th order fully implicit method based on a BDF-type Chebyshev approximation \cite{ramos_vigo:2007}:
\begin{align}\label{RK_ramos_vigo2007}
    \mathbf{A} = 
    \begin{bmatrix}
        \frac{22-\sqrt{2}}{96} & \frac{5-8\sqrt{2}}{48} & \frac{22-7\sqrt{2}}{96} & \frac{-1}{16} \\[4pt]
        \frac{4+3\sqrt{2}}{24} & \frac{1}{6} & \frac{4-3\sqrt{2}}{24} & 0 \\[4pt]
        \frac{22+7\sqrt{2}}{96} & \frac{5+8\sqrt{2}}{48} & \frac{22+\sqrt{2}}{96} & \frac{-1}{16} \\[4pt]
        \frac{1}{3} & \frac{1}{3} & \frac{1}{3} & 0
    \end{bmatrix} \quad\textrm{and}\quad
    \mathbf{b} = \begin{bmatrix} \frac{1}{3} & \frac{1}{3} & \frac{1}{3} & 0 \end{bmatrix}^T.
\end{align}
By defining \(\beta := \cos(\alpha)\), the generalized \(E\)-polynomial \eqref{Def:genEpoly} becomes the following bivariate polynomial in \(y\) and \(\beta\):
\begin{align}
    E(y^2; \beta) = y^2 \Big( &y^{14} + 22\beta y^{12} + (272\beta^2 - 16)y^{10} + (1920\beta^3 + 96\beta)y^8 \nonumber \\
    &+ (6144\beta^4 + 3840\beta^2)y^6 + (36864\beta^3 + 10752\beta)y^4 + 73728\beta y^2 + 294912\beta \Big) \, .
\end{align}
When $\beta = 0$ (corresponding to $\alpha = \frac{\pi}{2}$), the $E$-polynomial 
\begin{align*}
    E\left(y; \frac{\pi}{2}\right) = y^{6} \, (y^2 - 16) \, ,
\end{align*}
is negative for $|y| < 4$, and the scheme is not $A$-stable.

Instead, we determine a lower bound for the maximal value of \(\alpha\) for which \(E(y^2; \beta) \geq 0\). The last \(\beta\) value before reaching an empty feasible set is the bound
\begin{align}\label{Eq:VR_BetaOpt}
    \beta^*=\frac{19699132}{4466212691} \qquad  (\alpha^*\approx 89.74728^o) \, .  
\end{align}

Solving Problem \eqref{opt:SDP_Epoly} with the value $\beta^*$ (details provided in Appendix~\ref{Appendix:Ramos}) produces the positive certificate
\begin{align*}
    E(y^2;\beta^*)=y^2\vec{y}^T\mat{F}(\vec{\bar{\eta}})\vec{y}\geq0
\end{align*}
where $\mat{F}(\vec{\bar{\eta}})=\L\D\L^T$, and the diagonal matrix $\D\succ0$, certifying the scheme is $A(\alpha)$-stable for maximal angle $\alpha\geq\alpha^*$. 

%==================================================================================================
\subsubsection{The ESDIRK(8,6) Scheme of Skvortsov}\label{subsec:Skvortsov}
%==================================================================================================
We examine the 8-stage, 6th order ESDIRK scheme of Skvortsov \cite{Skvortsov2006} characterized by its Butcher matrix:
\begin{equation}\label{eq:skvortsov}
    \renewcommand\arraystretch{1.5}
    \setlength\arraycolsep{2pt}
    \A =\left[\begin{array}{ccccccccc}
    0 & & & & & & & &\\
    \frac{1}{6} & \frac{1}{6} & & & & & & &\\
    \frac{11}{96} & \text{-}\frac{1}{32} & \frac{1}{6} & & & & & &\\
    \frac{1}{12} & \text{-}\frac{1}{4} & \frac{1}{2} & \frac{1}{6} & & & & &\\
    \text{-}\frac{2015}{15072} & \text{-}\frac{6987}{5024} & \frac{3271}{1884} & \frac{175}{471} & \frac{1}{6} & & & &\\
    \text{-}\frac{326531}{573678} & \text{-}\frac{114988}{31871} & \frac{1208156}{286839} & \frac{132950}{286839} & \frac{68}{203} & \frac{1}{6} & & &\\
    \text{-}\frac{331717945}{2106545616} & \text{-}\frac{480525599}{416107776} & \frac{2240951089}{1404363744} & \frac{394951619}{2808727488} & \text{-}\frac{5160553}{26834976} & \frac{35815}{352512} & \frac{1}{6} & &\\
    \frac{16264655341}{73026914688} & \frac{9786099235}{14425069568} & \text{-}\frac{34306812733}{48684609792} & \text{-}\frac{15985588007}{97369219584} & \frac{37652437}{930279168} & \text{-}\frac{340747}{12220416} & \frac{1}{26} & \frac{1}{6} &\\
    \frac{7}{90} & 0 & 0 & 0 & \frac{16}{45} & \text{-}\frac{4}{45} & \frac{2}{15} & \frac{16}{45} & \frac{1}{6}
    \end{array}\right]
\end{equation}
The scheme is stiffly accurate, with $\b^T$ being the last row of $\A$.

The generalized $E$-polynomial, in $y$ and $\beta$, takes the form:
\begin{align*}
    E(y^2,\beta)&=y^2 \sum_{j=0}^{15} q_{2j}(\beta) \, y^{2j} \,  \qquad \textrm{where} \qquad \beta = \cos(\alpha) \, .
\end{align*}
Appendix~\ref{Appendix:Skvortsov} presents formulas for the polynomials $q_{2j}(\beta)$.  

Similar to the scheme of Ramos and Vigo, this scheme fails to be $A$-stable. When $\beta = 0$ ($\alpha = \frac{\pi}{2}$), the $E$-polynomial simplifies to:
\begin{align*}
    E\left(y ; \frac{\pi}{2}\right)= y^8 \left(y^8 + \tfrac{61415271}{616225}y^6 -\tfrac{91554624}{3925}y^4 
    -\tfrac{3175034112}{3925}y^2 -\tfrac{4152010752}{785} \right) \, .
\end{align*}
Since the lowest order monomial term is negative, e.g., $-4152010752/785 < 0$, the polynomial $E(y; \pi/2)$ is negative for values of $y$ near the origin. 

After iterating through values of $\beta$, we establish the bound:
\begin{align}\label{Eq:BetaOptSkv}
    \beta^*=\frac{2218472195}{100000000000} \qquad  (\alpha^*\approx 88.7288^o) \, .
\end{align}
This value is consistent with the $\alpha$ bound reported in \cite{Skvortsov2006} and is now accompanied by a rigorous certificate. 

The $E$-polynomial produced by $\beta^*$ can be written as 
\begin{align*}%\label{Eq:EPolySkvortsov}
    E(y^2; \beta^*) = y^2 \vec{y}^T \mat{F}(\veta) \vec{y} \,,  \quad \textrm{with}\quad
    \mat{F}(\veta) = \P + \sum_{j = 1}^{105} \eta_{j} \Nb_j \, , 
\end{align*}
where $\P \in \mathbb{Q}^{16\times 16}$ is a diagonal matrix containing three negative coefficients and $d = 105$ is the dimension \eqref{eq:d}.  Numerically solving Problem \eqref{opt:SDP_Epoly} yields a solution $\veta^*$, which again we round for simplicity to a nearby rational value $\bar{\veta}$ presented in Appendix~\ref{Appendix:Skvortsov}. Substituting $\bar{\veta}$ into the LMI and factorizing yields 
\begin{align}\label{Eq:EPolyFactorizedSkvortsov}
    E(y^2;\beta^*)=y^2 \vec{y}^T \L \D \L^T \vec{y} \geq 0 \, , 
\end{align}
where $\D, \L \in \mathbb{Q}^{16 \times 16}$ with $\D$ being diagonal with non-negative entries. 
\begin{remark}
    Even though $\veta^*$ has been rounded to a nearby rational value $\bar{\veta}$ with fewer digits, which we report in the Appendix, the matrix and $\D$ in \eqref{Eq:EPolyFactorizedSkvortsov} contains integer denominators with up to 548 digits. \myremarkend
\end{remark}

%================================================================================    
\section{Discussion and Conclusions}\label{sec:Conclusion}
%================================================================================

This paper adopts a semidefinite optimization approach to rigorously certify $A$- and $A(\alpha)$- stability in Runge-Kutta schemes and improves the algebraic conditions for $A$-stability by incorporating the order conditions. A key application of this work is the validation and construction of stable time-stepping schemes for potential implementation in industrial software. Looking ahead, these approaches and perspectives could also be useful in certifying stability in other time-integration schemes. Semidefinite programming can be used to certify stability in settings where algebraic conditions for stability are known, such as the algebraic stability of general linear methods. On the other hand, there are settings, such as $A$-stability in general linear methods, where, to our knowledge, algebraic conditions for stability have yet to be formulated, providing opportunities for future work.  

%================================================================================    
\section*{Acknowledgments}
%================================================================================
The authors would like to thank David Ketcheson for helpful discussions.  This material is based upon work supported by the National Science Foundation under Grants No.\ DMS--2012268 and DMS--2309727. 

\section*{Data Availability}
Certificates of stability and the supporting numerical code are available online from 

\noindent\href{https://github.com/ajuhl/Algebraic-Conditions-for-Stability-Certified-via-SDP}{\texttt{https://github.com/ajuhl/Algebraic-Conditions-for-Stability-Certified-via-SDP}}.

\printbibliography

@book{HairerWanner1996,
  title={Solving {O}rdinary {D}ifferential {E}quations {II}(2nd Revised. Ed.): {S}tiff and {D}ifferential-{A}lgebraic {P}roblems},
  author={Hairer, E. and Wanner, G.},
  year={1996},
  publisher={Springer Berlin Heidelberg}
}

@book{HairerNorsettWanner1993,
	author	= {Hairer, E. and N{\o}rsett, S. P. and Wanner, G.},
	title		= {Solving {O}rdinary {D}ifferential {E}quations {I} (2nd Revised. Ed.): {N}onstiff {P}roblems},
	year		= {1993},
	isbn		= {0-387-56670-8},
	publisher	= {Springer-Verlag, New York}
}

@ARTICLE{scherer_turke:1989,
  AUTHOR =	 {Scherer, R. and Turke, H.},
  YEAR =	 1989,
  JOURNAL =	 {Applied Numerical Mathematics},
  TITLE =	 {Algebraic Characterization of A-stable Runge-Kutta Methods},
  PAGES =	 {133--144},
  VOLUME =	 {5}
}

@ARTICLE{scherer_wendler:1994,
  AUTHOR =	 {Scherer, R. and Wendler, W.},
  YEAR =	 1994,
  JOURNAL =	 {SIAM Journal on Numerical Analysis},
  TITLE =	 {Complete Algebraic Characterization of A-stable Runge-Kutta Methods},
  PAGES =	 {540--551},
  VOLUME =	 {31},
  NUMBER = {2},
  PUBLISHER = {Society for Industrial and Applied Mathematics}
}

@ARTICLE{cooper:1986,
  AUTHOR =	 {Cooper, G.},
  YEAR =	 1986,
  JOURNAL =	 {Numerical Analysis},
  TITLE =	 {An algebraic condition for A-stable Runge-Kutta methods},
  PAGES =	 {32--46},
  PUBLISHER = {Longman Scientific & Technical}
}

@ARTICLE{ramos_vigo:2007,
  AUTHOR =	 {Ramos, H. and Vigo-Aguiar, J.},
  YEAR =	 2007,
  JOURNAL =	 {Journal of Computational and Applied Mathematics},
  TITLE =	 {A fourth-order Runge-Kutta method based on BDF-type Chebyshev approximations},
  VOLUME=   {204},
  PAGES =	 {124--136},
  PUBLISHER = {Elsevier}
}

@article{Butcher1977,
    author = {Butcher, J.C.},
    title = {On {A}-stable implicit {R}unge-{K}utta methods},
    journal = {BIT Numerical Mathematics},
    year = {1977},
    volume = {17},
    issue = {4},
    pages = {375--378}
}

@article{BiswasKetchesonSeiboldShirokoff2024,
    author = {Biswas, A. and Ketcheson, D. and Seibold, B. and Shirokoff, D.},
    title = {Algebraic Structure of the Weak Stage Order Conditions for Runge–Kutta Methods},
    journal = {SIAM Journal on Numerical Analysis},
    volume = {62},
    number = {1},
    pages = {48-72},
    year = {2024},
    doi = {10.1137/22M1483943}
}

@article{Norsett1975,
    author = {N{\o}rsett, S.P.},
    title = {C-polynomials for rational approximations to the exponential function},
    journal = {Numer. Math.},
    year = {1975},
    volume = {25},        
    pages = {39--56}
}

@article{Hairer1982,
    author = {Hairer, E.},
    title = {Constructive characterization of {A}-stable approximations to exp z and its connection with algebraically stable {R}unge-{K}utta methods},
    journal = {Numer. Math.},
    volume = {39},
    pages = {247--258},
    year = {1982}
}

@article{HairerWanner1981,
    author = {Hairer, E. and Wanner, G.},
    title = {Algebraically stable and implementable {R}unge-{K}utta methods of high order},
    journal = {SIAM J. Numer. Anal.},
    volume = {18},
    pages = {1098--1108},
    year = {1981}
}

@misc{AlamriKetcheson2024,
    author = {Alamri, Y. and Ketcheson, D.},
    title = {Very high-order {A}-stable stiffly accurate diagonally implicit {R}unge-{K}utta methods with error estimators},
    year = {2024}, 
    note = {arXiv: 2211.14574} 
}

@article{BiswasKetchesonSeiboldShirokoff2023,
    title = {Design of {DIRK} schemes with high weak stage order},
    author = {Biswas, A. and Ketcheson, D. and Seibold, B. and Shirokoff, D.},
    journal = {Commun. Appl. Math. Comput.},
    volume = {18},
    pages = {1--28},
    year = {2023}    
}

@misc{BiswasKetchesonRobertsSeiboldShirokoff2024,
    title = {Explicit Runge Kutta Methods that Alleviate Order Reduction},
    author = {Biswas, A. and Ketcheson, D. and Roberts, S. and Seibold, B. and Shirokoff, D.},
    year = {2023},
    note = {arXiv:2310.02817}
}

@article{Skvortsov2006,
    author ={Skortsov, L.M.},
    title = {Diagonally implicit Runge-Kutta methods for stiff problems},
    journal = {Comput. Math. and Math. Phys.},
    volume = {46},
    pages = {2110-2123},
    year = {2006}
}

@techreport{dahlquist1975stability,
  title={On stability and error analysis for stiff non-linear problems part i},
  author={Dahlquist, Germund},
  year={1975},
  institution={CM-P00069396}
}

@book{Lasserre2010,
    author = {J. B. Lasserre},
    title = {Moments, Positiv Polynomials and Their Applications},
    publisher = {Imperial College Press},
    year = {2010}
}

@misc{cvx,
  author       = {CVX Research, Inc.},
  title        = {{CVX}: Matlab Software for Disciplined Convex Programming, version 2.0},
  howpublished = {\url{http://cvxr.com/cvx}},
  month        = {8},
  year         = {2012}
}

@incollection{gb08,
  author    = {M. Grant and S. Boyd},
  title     = {Graph implementations for nonsmooth convex programs},
  booktitle = {Recent Advances in Learning and Control},
  series    = {Lecture Notes in Control and Information Sciences},
  editor    = {V. Blondel and S. Boyd and H. Kimura},
  publisher = {Springer-Verlag Limited},
  pages     = {95--110},
  year      = {2008}  
}

@article{JinLavaei2020,
    author = {M. Jin and J. Lavaei},
    title = {Stability-Certified Reinforcement Learning: A Control-Theoretic Perspective},
    journal = {IEEE Access},
    volume = {8},    
    pages = {229086},
    year = {2020}
}

@article{DroriTeboulle2014,
    author = {Y. Drori and M. Teboulle},
    title = {Performance of first-order methods for smooth convex minimization: a novel approach},
    journal = {Mathematical Programming},    
    volume = {145},
    year = {2014},    
    pages = {1436--4646},
    doi = {10.1007/s10107-013-0653-0}    
}

@misc{Grimmer2024,
    title = {Provably Faster Gradient Descent via Long Steps}, 
    author = {B. Grimmer}, 
    note = {arxiv.org/abs/2307.06324},
    year = {2024}
}

@inproceedings{MajumdarAhmadiTedrake2013,
  author={Majumdar, Anirudha and Ahmadi, Amir Ali and Tedrake, Russ},
  booktitle={2013 IEEE International Conference on Robotics and Automation}, 
  title={Control design along trajectories with sums of squares programming}, 
  year={2013},
  volume={},
  number={},
  pages={4054-4061}
}

@INPROCEEDINGS{GahlawatValmorbida2017,
  author={A. Gahlawat and G. Valmorbida},
  booktitle={2017 IEEE 56th Annual Conference on Decision and Control (CDC)}, 
  title={A semi-definite programming approach to stability analysis of linear partial differential equations}, 
  year={2017},
  pages={1882-1887}
}

@article{YaoZhangZhou2001,
    author = {Yao, David D. and Zhang, Shuzhong and Zhou, Xun Yu},
    title = {Stochastic Linear-Quadratic Control via Semidefinite Programming},
    journal = {SIAM Journal on Control and Optimization},
    volume = {40},
    number = {3},
    pages = {801-823},
    year = {2001}
}

@INPROCEEDINGS{AhmadiJungers2013,
  author={Ahmadi, Amir Ali and Jungers, Raphaël M.},
  booktitle={52nd IEEE Conference on Decision and Control}, 
  title={Switched stability of nonlinear systems via SOS-convex Lyapunov functions and semidefinite programming}, 
  year={2013},
  volume={},
  number={},
  pages={727-732}
}

@article{Dahlquist1963,
    author = {G. Dahlquist},
    title = {A special stability problem for linear multistep methods}, 
    journal = {BIT Numerical Mathematics},
    volume = {3},
    year = {1963}, 
    pages = {27--43}
}

@incollection{Dahlquist1978,
    author = {G. Dahlquist},
    title = {Positive Functions and some Applications to Stability Questions for Numerical Methods},
    editor = {CARL {DE BOOR} and GENE H. GOLUB},
    booktitle = {Recent Advances in Numerical Analysis},
    publisher = {Academic Press},
    pages = {1-29},
    year = {1978}
}

@article{Wanner1980,
    author = {G. Wanner},
    year = {1980},
    journal = {BIT Numerical Mathematics},
    title = {Characterization of allA-stable methods of order 2m-4},
    volume = {20},
    pages = {367--374}
}

\clearpage

\appendix
\renewcommand{\thesection}{Appendix \Alph{section}:}
\renewcommand{\thesubsection}{\Alph{section}.\arabic{subsection}}

%==================================================================================================
\section{Supplemental Details for \(A\)-stable Schemes}
%==================================================================================================

Here, we present numerical coefficients verifying $A$-stability.

%===============================================
\subsection{DIRK (13,8)[1]A[(14,6)]A}\label{Appendix:DIRK13_8}
%===============================================
The scheme has $E$-polynomial $E(y) = y^2 \vec{y}^T \mat{F}(\vec{\bar{\eta}}) \vec{y}$ where 
\begin{align}
    \mat{F}(\vec{\bar{\eta}}) = \mat{P} + \sum_{j=1}^{66} \leta_{j} \mat{N}_j \succeq 0 \, ,
\end{align}
with coefficients of $\vec{\bar{\eta}}$ given by:
\begin{alignat*}{3}
    \leta_1 &= -8470700 & \qquad \leta_{23} &= -2465700 & \qquad \leta_{45} &= \phantom{-}16400 \\ 
    \leta_2 &= \phantom{-}0 & \qquad \leta_{24} &= \phantom{-}665333207306300 & \qquad \leta_{46} &= \phantom{-}4362777200 \\
    \leta_3 &= -3700 & \qquad \leta_{25} &= -639200 & \qquad \leta_{47} &= \phantom{-}0 \\
    \leta_4 &= -17219137300 & \qquad \leta_{26} &= -27185350128356800 & \qquad \leta_{48} &= \phantom{-}2843900 \\
    \leta_5 &= -2500 & \qquad \leta_{27} &= -58914800 & \qquad \leta_{49} &= -1564541946300 \\
    \leta_6 &= \phantom{-}1552662793500 & \qquad \leta_{28} &= \phantom{-}540480365078400 & \qquad \leta_{50} &= -4500 \\
    \leta_7 &= \phantom{-}1100 & \qquad \leta_{29} &= -3060809665186300 & \qquad \leta_{51} &= \phantom{-}338036200 \\
    \leta_8 &= -291771102600 & \qquad \leta_{30} &= \phantom{-}933800 & \qquad \leta_{52} &= \phantom{-}26403742783600 \\
    \leta_9 &= \phantom{-}3000 & \qquad \leta_{31} &= -101765843690800 & \qquad \leta_{53} &= -61100 \\
    \leta_{10} &= \phantom{-}5801136200 & \qquad \leta_{32} &= \phantom{-}0 & \qquad \leta_{54} &= -20078540400 \\
    \leta_{11} &= -10101540029400 & \qquad \leta_{33} &= -16794400 & \qquad \leta_{55} &= \phantom{-}0 \\
    \leta_{12} &= -949900 & \qquad \leta_{34} &= \phantom{-}0 & \qquad \leta_{56} &= -256500 \\
    \leta_{13} &= \phantom{-}45726880755500 & \qquad \leta_{35} &= \phantom{-}3600 & \qquad \leta_{57} &= \phantom{-}294637386500 \\
    \leta_{14} &= \phantom{-}244100 & \qquad \leta_{36} &= \phantom{-}19168235300 & \qquad \leta_{58} &= \phantom{-}600 \\
    \leta_{15} &= -2222445564500 & \qquad \leta_{37} &= \phantom{-}0 & \qquad \leta_{59} &= -63663500 \\
    \leta_{16} &= -3000 & \qquad \leta_{38} &= -14572100 & \qquad \leta_{60} &= -1286411886200 \\
    \leta_{17} &= -1556643190544700 & \qquad \leta_{39} &= \phantom{-}0 & \qquad \leta_{61} &= \phantom{-}200 \\
    \leta_{18} &= \phantom{-}1744900 & \qquad \leta_{40} &= \phantom{-}18000 & \qquad \leta_{62} &= \phantom{-}978576800 \\
    \leta_{19} &= \phantom{-}292506767260200 & \qquad \leta_{41} &= \phantom{-}16836266900 & \qquad \leta_{63} &= \phantom{-}0 \\
    \leta_{20} &= -1976300 & \qquad \leta_{42} &= \phantom{-}0 & \qquad \leta_{64} &= -5871439100 \\
    \leta_{21} &= -5815692149200 & \qquad \leta_{43} &= -3637800 & \qquad \leta_{65} &= \phantom{-}0 \\
    \leta_{22} &= -13689040894555700 & \qquad \leta_{44} &= -5738476103900 & \qquad \leta_{66} &= \phantom{-}1269200 \\
\end{alignat*}

\begin{landscape}
%===============================================
\subsection{The Perturbed Schemes from \S~\ref{subsubsection:PertSchemes}}\label{Appendix:PertSchemes}
%===============================================
In this appendix, we report the coefficients of the perturbed schemes $(\mat{\tilde{A}}, \vec{\tilde{b}})$ for which the lemmas in \S~\ref{subsubsection:PertSchemes} were established.

\bigskip

\noindent 
{\bf SDIRK(9,6)[1]SAL[(9,5)A]}: The perturbed scheme coefficients are
\begin{align*}
    \renewcommand\arraystretch{1.5}
    \setlength\arraycolsep{2pt}
    \begin{aligned}
        \mat{\tilde{A}}&=\left[\begin{array}{ccccccccc}
        \frac{87518253}{401224696} & & & & & & & &\\
        \text{-}\frac{109147862}{1208036163} & \frac{87518253}{401224696} & & & & & & &\\
        \frac{70447391}{407323275} & \text{-}\frac{60128027}{170018875} & \frac{87518253}{401224696} & & & & & &\\
        \frac{258011928}{503929669} & \frac{32776647}{1131632696} & \text{-}\frac{13636148}{946751265} & \frac{87518253}{401224696} & & & & &\\
        \frac{2139251}{459753907} & \text{-}\frac{18361775}{242765601} & \frac{13889605}{63926963} & \text{-}\frac{17789601}{861400843} & \frac{87518253}{401224696} & & & &\\
        \frac{48479320}{54097599} & \frac{222681723}{1598951647} & \text{-}\frac{6683180}{35754039} & \frac{14834219}{220428796} & \text{-}\frac{67840169}{193336343} & \frac{87518253}{401224696} & & &\\
        \frac{67080581}{121311880} & \text{-}\frac{401007739}{912707597} & \frac{169517869}{507988720} & \text{-}\frac{10552416}{310889555} & \text{-}\frac{54396621}{357996284} & \frac{12212839}{571158716} & \frac{87518253}{401224696} & &\\
        \frac{176730716}{279920507} & \frac{185311713}{255696311} & \text{-}\frac{40314204}{93283073} & \frac{92464054}{154464243} & \text{-}\frac{281536253}{397040384} & \text{-}\frac{15560941}{32151589} & \frac{170501635}{450595763} & \frac{87518253}{401224696} &\\
        0 & \tilde{b}_2 & \tilde{b}_3 & \tilde{b}_4 & \tilde{b}_5 & \tilde{b}_6 & \tilde{b}_7 & \tilde{b}_8 & \frac{87518253}{401224696}
        \end{array}\right]
    \end{aligned}
\end{align*}
where $\tilde{b}_j$ ($3 \leq j \leq 8$) is provided by solving for the tall tree conditions \eqref{eq:ordercondition} (for $p = 6$), leaving a free variable $\tilde{b}_2$ which is selected to minimize the $\ell_2$ difference $\b-\vec{\tilde{b}}$. 

\bigskip

\newpage 
\noindent
{\bf WSO DIRK(12,5,5)} The perturbed scheme coefficients are
\[
\renewcommand\arraystretch{1.75}
\setlength\arraycolsep{1pt}
\mat{\tilde{A}}=
{
\left[\begin{array}{cccccccccccc}
\frac{10747729}{261281103} & & & & & & & & & & &\\
\frac{39255733}{244819013} & \frac{19264472}{289086473} & & & & & & & & & &\\
\text{-}\frac{63980223}{186836899} & \frac{70601555}{81544818} & \frac{48699329}{492234648} & & & & & & & & &\\
\frac{757656751}{80284215} & \text{-}\frac{5452955845}{500830197} & \frac{261637874}{98954371} & \frac{86994158}{471217857} & & & & & & & &\\
\text{-}\frac{290059410}{846787661} & \frac{142027951}{274596637} & \frac{63720572}{69536691} & \frac{8719292}{166871841} & \frac{230505997}{1977768146} & & & & & & &\\
\text{-}\frac{96530823}{46089059} & \frac{278501442}{108044467} & \frac{186959114}{327745145} & \frac{31757051}{261668409} & \text{-}\frac{65608216}{138056009} & \frac{50822223}{96152122} & & & & & &\\
\frac{110645970}{326232259} & \text{-}\frac{52198210}{186593643} & \frac{48069176}{46243347} & \frac{36028733}{602611029} & \text{-}\frac{42230947}{197997752} & \frac{11234921}{134641567} & \frac{339062341}{1406835502} & & & & &\\
\frac{432515173}{73254485} & \frac{1037694284}{327224921} & \text{-}\frac{1241126819}{100347987} & \text{-}\frac{76295713}{152911958} & \frac{155195477}{71832145} & \frac{268359096}{140054533} & \frac{309243168}{155550259} & \frac{266682747}{1194765721} & & & &\\
\frac{131193951}{284188360} & \text{-}\frac{74904386}{387416395} & \text{-}\frac{136922649}{1129220324} & \frac{46345993}{695639065} & \frac{61016892}{143403385} & \frac{74333617}{94618599} & \frac{157610940}{188314681} & \frac{31846751}{198449271} & \frac{305893355}{845914548} & & &\\
\text{-}\frac{82024283}{115728139} & \frac{91013047}{140744863} & \frac{63214225}{132835881} & \text{-}\frac{178521351}{694495505} & \frac{96297873}{85736426} & \frac{97069417}{174996915} & \frac{25465272}{79767817} & \frac{11384038}{31516593} & \frac{232759857}{396742103} & \frac{83426711}{354434193} & &\\
\frac{100055236}{234642175} & \frac{213090564}{161088509} & \frac{96986289}{228435568} & \text{-}\frac{174464145}{68947184} & \text{-}\frac{67238506}{859605737} & \frac{2068579853}{1961737581} & \frac{170367931}{366730407} & \frac{198092237}{172991608} & \frac{312091183}{725567705} & \frac{159786147}{106558690} & \frac{9565123}{660600961} &\\
\frac{5486027}{454369097} & \frac{9757227}{18810635} & \frac{62278071}{555407431} & \text{-}\frac{1694527}{341651848} & \text{-}\frac{133046372}{98916929} & \frac{323864293}{952870301} & \frac{261379716}{320347663} & \tilde{b}_8 & \tilde{b}_9 & \tilde{b}_{10} & \tilde{b}_{11} & \tilde{b}_{12}
\end{array}\right]}
\]
The coefficients $\tilde{b}_j \in \mathbb{Q}$ ($8 \leq j \leq 12$) are chosen to satisfy the order conditions exactly \eqref{eq:ordercondition} for $p = 5$.

The $E$-polynomial $E(y) = y^2 \vec{y}^T \F(\vec{\leta})\vec{y} = \L \D \L^T \geq 0$ where $\vec{\leta}\cdot10^{-3}$ has coefficients:
\begin{alignat*}{6}
    \leta_{1} &= -6922561820555       &\qquad \leta_7 &= \phantom{-}3934894             &\qquad \leta_{13} &= -550                      &\qquad \leta_{19} &= -40414145977      &\qquad \leta_{25} &= \phantom{-}692417          &\qquad \leta_{31} &= -10          \\
    \leta_2 &= -6159041               & \leta_8 &= -6775128853059                      & \leta_{14} &= -1017731680875                  & \leta_{20} &= -57578                 &   \leta_{26} &= -2928211                     & \leta_{32} &= -40848              \\
    \leta_3 &= \phantom{-}401988060958 &  \leta_9 &= -4571676                           & \leta_{15} &= -1687736                        & \leta_{21} &= \phantom{-}139245638   &   \leta_{27} &= -2                            & \leta_{33} &= \phantom{-}0           \\
    \leta_4 &= \phantom{-}209390       &  \leta_{10} &= \phantom{-}142233029108         & \leta_{16} &= \phantom{-}8353405937           & \leta_{22} &= \phantom{-}147         & 	\leta_{28} &= -5711                          & \leta_{34} &= \phantom{-}859      \\   
    \leta_5 &= -3199255341            & \leta_{11} &= \phantom{-}198557                & \leta_{17} &= \phantom{-}14765                & \leta_{23} &= -553788073             &   \leta_{29} &= \phantom{-}3                  & \leta_{35} &= \phantom{-}0           \\ 
    \leta_6 &= -4095                  &   \leta_{12} &= -490076976                      & \leta_{18} &= -10433650                       & \leta_{24} &= -636                   &   \leta_{30} &= \phantom{-}143788             & \leta_{36} &= -3
\end{alignat*}
\end{landscape}

%==================================================================================================
\section{Supplemental Details for \(A(\alpha)\)-stable Schemes}
%==================================================================================================
\subsection{The IRK(4,4) Scheme of Ramos and Vigo}\label{Appendix:Ramos}
The scheme has $E$-polynomial $E(y^2,\beta^*) = y^2 \vec{y}^T \mat{F}(\vec{\bar{\eta}}) \vec{y}$ where 
\begin{align}
    \mat{F}(\vec{\bar{\eta}}) = \mat{P} + \sum_{j=1}^{21} \leta_{j} \mat{N}_j \succeq 0 \, ,
\end{align}
with coefficients of $\vec{\bar{\eta}}$ given by $\leta_2=\leta_4=\leta_6=\leta_8=\leta_{10}=\leta_{13}=\leta_{15}=\leta_{17}=\leta_{20}=0$ and:
\begin{alignat*}{4}
\leta_1&=-\tfrac{343818785}{387257} &\qquad \leta_7&= -\tfrac{1002782638}{963823} &\qquad \leta_{12}&= -\tfrac{140623753}{190944} &\qquad \leta_{18}&= \phantom{-}\tfrac{1407711}{121108}\\
\leta_3&=\phantom{-}\tfrac{44352332}{270307} &\leta_9&= \phantom{-}\tfrac{26195675}{165379} &\leta_{14}&= \phantom{-}\tfrac{19205029}{233487} &\leta_{19}&= -\tfrac{23169437}{338293}\\
\leta_5&=-\tfrac{7044484}{1620291} &\leta_{11}&= -\tfrac{526928}{268115} &\leta_{16}&=-\tfrac{55546025}{187031} &\leta_{21}&=-\tfrac{2727674}{410745}
\end{alignat*}

%==================================================
\subsection{ESDIRK(8,6) Skvortsov Scheme in \S\ref{subsec:Skvortsov}} \label{Appendix:Skvortsov}
%==================================================
We present coefficients for the scheme in subsection~\ref{subsec:Skvortsov}. The generalized $E$-polynomial is
\begin{align*}
    E(y,\beta)&=y^2 \sum_{j=0}^{15} q_{2j}(\beta) \, y^{2j} \,  \qquad \textrm{where} \qquad \beta = \cos(\alpha) \, , 
\end{align*}
with polynomial coefficients of $\beta$:
\begin{alignat*}{2}
    q_{30}(\beta) &= 1 && \\
    q_{28}(\beta) &= 96\beta\vphantom{\beta^1} && \\
    q_{26}(\beta) &= 4032\beta^2 +\tfrac{61415271}{616225} && \\
    q_{24}(\beta) &= 96768\beta^3 +\tfrac{73235232}{3925}\beta && \\
    q_{22}(\beta) &= 1451520\beta^4+\tfrac{3567255552}{3925}\beta^2-\tfrac{91554624}{3925} &&\\
    q_{20}(\beta) &= 13934592\beta^5 +\tfrac{14120096256}{785}\beta^3+\tfrac{9673437312}{3925}\beta \\
    q_{18}(\beta) &= 83607552\beta^6 +\tfrac{158718486528}{785}\beta^4+\tfrac{70172863488}{785}\beta^2-\tfrac{3175034112}{3925} &&    \\
    q_{16}(\beta) &= 286654464\beta^7 +\tfrac{1149206704128}{785}\beta^5+\tfrac{985309774848}{785}\beta^3+\tfrac{136916137728}{785}\beta && \\    
    q_{14}(\beta) &= 429981696\beta^8 +\tfrac{5111514906624}{785}\beta^6+\tfrac{8045011279872}{785}\beta^4+\tfrac{694870576128}{157}\beta^2-\tfrac{4152010752}{785} &&\\
    q_{12}(\beta) &= \tfrac{10416951558144}{785}\beta^7+\tfrac{41311620145152}{785}\beta^5+\tfrac{34328744275968}{785}\beta^3 +\tfrac{5012232804864}{785}\beta &&\\
    q_{10}(\beta) &= \tfrac{650132324352}{5}\beta^6+232190115840\beta^4 +121899810816\beta^2 &&\\
    q_{8}(\beta) &= \tfrac{3064909529088}{5}\beta^5+835884417024\beta^3 +121899810816\beta  &&\\
    q_6(\beta) &= 2298682146816\beta^4 +1671768834048\beta^2 &&\\
    q_4(\beta) &= 6269133127680\beta^3 +1253826625536\beta &&\\
    q_2(\beta) &= 9403699691520\beta^2 &&\\    
    q_0(\beta) &= 5642219814912\beta \, . 
\end{alignat*}

\bigskip
Solution coefficients $\vec{\bar{\eta}}$ for non-negativity of $E(y, \beta^*)$. 
\begin{alignat*}{4}
 \leta_{1}&=-\tfrac{544417542815}{1496} &\qquad\leta_{28}&=-\tfrac{12225979229715}{323} &\qquad\leta_{55}&=-\tfrac{3023415560208}{431} &\qquad\leta_{82}&=\phantom{-}\tfrac{2813400132854}{117}\\
 \leta_{3}&=\phantom{-}\tfrac{29842486335}{2687} &\leta_{30}&=\phantom{-}\tfrac{433008073700}{13} &\leta_{57}&=\phantom{-}\tfrac{667361119609}{1153} &\leta_{84}&=-\tfrac{958434399311}{491}\\
\leta_{5}&=-\tfrac{49740795}{931} &\leta_{32}&=-\tfrac{8424611000191}{98} &\leta_{59}&=-\tfrac{15606394145}{932} &\leta_{85}&=\phantom{-}\tfrac{11434535023}{204}\\
\leta_{7}&=-\tfrac{28918182864}{1151} &\leta_{34}&=-\tfrac{25273367886229}{395} &\leta_{61}&=\phantom{-}\tfrac{70764854}{907} &\leta_{87}&=-\tfrac{220302170}{863}\\
\leta_{9}&=\phantom{-}\tfrac{699763151}{1893} &\leta_{36}&=\phantom{-}\tfrac{9262966388511}{292} &\leta_{63}&=-\tfrac{6003091166555}{1293} &\leta_{89}&=\phantom{-}\tfrac{9715693027109}{1173}\\
\leta_{11}&=-\tfrac{611431235}{606} &\leta_{38}&=-\tfrac{11180085898009}{131} &\leta_{65}&=\phantom{-}\tfrac{145331883768}{577} &\leta_{91}&=-\tfrac{261194717165}{607}\\
\leta_{13}&=\phantom{-}\tfrac{4436247}{898} &\leta_{40}&=-\tfrac{10519983461815}{799} &\leta_{67}&=-\tfrac{2769061905}{757} &\leta_{93}&=\phantom{-}\tfrac{5649757805}{927}\\
\leta_{15}&=-\tfrac{20553023}{989} &\leta_{42}&=-\tfrac{13505692581791}{106} &\leta_{69}&=\phantom{-}\tfrac{2674611184492}{349} &\leta_{95}&=-\tfrac{5175464369937}{284}\\
\leta_{17}&=-\tfrac{184449}{1142} &\leta_{44}&=\phantom{-}\tfrac{13635317782915}{474} &\leta_{70}&=-\tfrac{435375081998}{497} &\leta_{96}&=\phantom{-}\tfrac{1361933288513}{864}\\
\leta_{19}&=-\tfrac{5040480296101}{64} &\leta_{46}&=-\tfrac{22015444860842}{179} &\leta_{72}&=\phantom{-}\tfrac{28638357917}{941} &\leta_{98}&=-\tfrac{161915093932}{3421}\\
\leta_{21}&=-\tfrac{20253725215145}{857} &\leta_{48}&=-\tfrac{23644956443647}{379} &\leta_{74}&=-\tfrac{18964563}{119} &\leta_{100}&=\phantom{-}\tfrac{405541643}{1810}\\
\leta_{23}&=\phantom{-}\tfrac{240883590679}{566} &\leta_{50}&=-\tfrac{2159073125275}{1597} &\leta_{76}&=-\tfrac{883878192398}{307} &\leta_{102}&=-\tfrac{4892528411838}{1501}\\
\leta_{25}&=\phantom{-}\tfrac{4696001199805}{754} &\leta_{51}&=\phantom{-}\tfrac{58606387358}{645} &\leta_{78}&=\phantom{-}\tfrac{38657713081}{326} &\leta_{103}&=\phantom{-}\tfrac{127826367589}{733}\\
\leta_{27}&=-\tfrac{26257051251763}{262} &\leta_{53}&=-\tfrac{2209952042}{1439} &\leta_{80}&=-\tfrac{1721374819}{1269} &\leta_{105}&=-\tfrac{1622486899}{642}
\end{alignat*}
All $\leta_j=0$ for any $j=1,\ldots,105$ not defined above.
\end{document}